\documentclass[preprint,12pt]{elsarticle}
\usepackage{amsmath}
\usepackage{amssymb}
\usepackage{amsthm, amscd}
\newtheorem{thm}{Theorem}[section]
\newtheorem{prop}[thm]{Proposition}
\newtheorem{lem}[thm]{Lemma}
\newtheorem{cor}[thm]{Corollary}
\newtheorem{defn}[thm]{Definition}
\newtheorem{fact}[thm]{Fact}
\theoremstyle{remark}
\newtheorem{rem}[thm]{Remark}

\newtheorem{exmp}[thm]{Example}
\newtheorem{prob}{Problem}

\newcommand{\mbb}{\mathbb}
\newcommand{\fra}{\mathfrak}
\newcommand{\p}{\fra{p}}
\newcommand{\q}{\fra{q}}

\newcommand{\ov}{\overline}

\newcommand{\N}{\mathbb{N}}
\newcommand{\Z}{\mbb{Z}}
\newcommand{\Q}{\mbb{Q}}
\newcommand{\D}{\mathcal{D}}

\newcommand{\isg}{Is(G'\cdot Z(G))}

\newcommand{\ds}{\displaystyle}

\journal{Annals of Pure and Applied Logic}
\begin{document}

\begin{frontmatter}



\title{Elementary coordinatization of finitely generated nilpotent groups}


\author[Alexei]{Alexei G. Myasnikov}

\address[Alexei]{Department of Mathematical Sciences,
Stevens Institute of Technology,
Castle Point on Hudson,
Hoboken, NJ 07030, USA, Tel:1-(201)-216-8598}

\ead{amiasnikov@gmial.com}

\author[Mahmood]{Mahmood Sohrabi\corref{Mahmood1}}

\address[Mahmood]{Department of Mathematical Sciences,
 Stevens Institute of Technology,
Castle Point on Hudson,
Hoboken, NJ 07030, USA, Tel:1-(201)-216-5425}
\ead{msohrab1@stevens.edu}
\cortext[Mahmood1]{Corresponding author}
 
\begin{abstract}
This paper has two main parts. In the first part we develop an elementary coordinatization for any nilpotent group $G$ taking exponents in a binomial principal ideal domain (PID) $A$. In case that the additive group $A^+$ of $A$ is finitely generated we prove using a classical result of Julia Robinson that one can obtain a central series for $G$ where the action of the ring of integers $\Z$ on the quotients of each of the consecutive terms of the series except for one very specific gap, called the special gap, is interpretable in $G$. Then we use a refinement of this central series to give a criterion for elementary equivalence of finitely generated nilpotent groups in terms of the relationship between group extensions and the second cohomology group.
\end{abstract}

\begin{keyword}
Nilpotent group\sep Elementary Equivalence \sep  Largest Ring of a Bilinear Map \sep Coordinatization \sep Abelian Deformation

\MSC 20 \sep 03C60 

\end{keyword}
\end{frontmatter}

\section{Introduction}
This paper continues the authors' efforts~\cite{MS, MS2010}, in providing a comprehensive and uniform approach to various model-theoretical questions on algebras and nilpotent groups. We introduce our main techniques for approaching the following fundamental problems (stated here for nilpotent groups).

\begin{prob}
Explain in algebraic terms when and why two finitely generated groups are elementarily equivalent.
\end{prob}
\begin{prob} Find an algebraic characterization of \emph{all} groups elementarily equivalent to a given \emph{arbitrary} finitely generated nilpotent group. \end{prob}

\begin{prob} Find an algebraic description of  nilpotent groups whose elementary theory is $\omega$-stable or of finite Morley rank.\end{prob}

\begin{prob} Given a finitely generated nilpotent group $G$, describe in algebraic terms a  set of axioms for the elementary theory $Th(G)$ of $G$.\end{prob}   

 Every finitely generated torsion-free nilpotent group $G$ can be looked at as a group admitting exponents in the ring of integers $\Z$ and therefore as a two-sorted structure $G_\Z=\langle G, \Z, s\rangle$, where $G$ is a group, $\Z$ is the ring of integers and $s$ is the predicate describing the action of $\Z$ on $G$. Obviously from logical point of view $G_\Z$ is a much richer structure than the pure group $G$. The main idea of our approach is to study to what extent $G_\Z$ can be recovered from $G$. We achieve this by interpreting a maximal ring $A_R(G)$ that acts on all quotients of a first-order definable central series $(R)$ of $G$, except for a special gap (we note that $A_R(G)$ is necessarily commutative). Indeed the $A_R(G)$-module structure of these modules are interpretable in $G$, as a pure group, uniformly with respect to the first-order theory $Th(G)$. Then the $\Z$-module structures of every quotient admitting an $A_R(G)$-module structure will be interpretable in $G$. The special gap mentioned above is the main cause of deviation of the elementarily equivalence from the isomorphism in the case of finitely generated nilpotent groups.
 The influence the special gap may have on the elementary equivalence and related questions is described completely in terms of second cohomology and abelian deformation of groups. The process of recovering the $\Z$-module structure of the quotients of terms of a central series of $G$ from the multiplication on $G$ using only the first-order theory of $G$ is termed {\em elementary coordinatization} of $G$.   
 
 In the first part of the paper (Sections 3-5) we describe the coordinatization of a finitely generated nilpotent group $G$ and show that (with exception of the special gap)  this coordinatization is first-order interpretable in $G$ uniformly with respect to $Th(G)$. In fact, the argument works as well for finitely generated nilpotent $A$-groups for an arbitrary binomial principal ideal domain $A$. It takes a particularly nice form if the additive group $A^+$ of $A$ is  finitely generated. Besides, we study the relationship between elementary equivalence and isomorphism of finitely generated $A$-modules in a two-sorted language (one sort for the abelian group and another for the ring of  scalars) where $A$ has a finitely generated additive group $A^+$. In particular, we give a description of elementarily equivalent finite dimensional $A$-algebras for such a ring $A$. In the second part (Sections 6-11) of the paper we address in full detail the first of the problems above, leaving the others for the future. We state the main results in Section~\ref{mainsec:intro}.

The problem of elementary equivalence of finitely generated nilpotent groups goes back to the results of Tarski where he showed that two free nilpotent groups of finite rank are elementarily equivalent if and only if they are isomorphic. In 1969, during his lectures at Novosibirsk University, Kargapolov asked whether this is the case for all finitely generated nilpotent groups. In 1971 Zilber solved this problem in the negative constructing particular examples of two finitely generated nilpotent groups of class 2 which are elementarily equivalent but not isomorphic \cite{Z71}. 

The first algebraic characterization of elementary equivalence of finitely generated (finite-by-) nilpotent groups was found by F. Oger~\cite{oger}. He proved that two such groups $G$ and $H$ are elementarily equivalent if and only if $G\times C \cong H \times C$, where $C$ refers to the infinite cyclic group written multiplicatively. This is a very nice result  which shows that in general elementary equivalence is not very far from isomorphism in the case of finitely generated nilpotent groups. Notice that this description does not reveal the algebraic reasons for the elementary equivalence of the groups $G$ and $H$. Indeed, this criterion enables one to check algorithmically (since the isomorphism problem is decidable here) if given $G$ and $H$ are elementarily equivalent, but it does not provide means to construct such an $H$ given a group $G$. The techniques developed in this paper allow us to answer this question completely in terms of second cohomologies and abelian deformations of groups. 
The techniques, approaches and aims (our main focus is on the Problems 1-4) are very different from Oger's \cite{oger}.


\section{Preliminaries, our approach and the main results}\label{mainsec:intro}
We open this section by discussing some preliminaries. In Subsection~\ref{approach:sec} we give an account of our approach and main techniques and state our main results.  
\subsection{Preliminaries on logic}
For the most part we follow standard model theory texts such as~\cite{hodges} regarding notation and model theory. A group $G$ is a structure with signature $\langle \cdot, ^{-1}, 1\rangle$ and the corresponding language is called $L$. 

For the convenience of the reader we shall introduce our notion of ``interpretation" which might be a bit different from the standard definitions.
\subsubsection{\protect Interpretations}\label{ss1}
\label{interpret1}Let $\mathfrak{B}$ and $\mathfrak{U}$ be  structures of signatures $\Delta$ and
$\Sigma$ respectively. We may assume that $\Sigma$ and $\Delta$ do not contain any function symbols replacing them if necessary with predicates (i.e. replacing operations with their graphs). The structure $\mathfrak{U}$ is said to be
\textit{interpretable}\index{interpretable} in $\mathfrak{B}$ with parameters $\bar{b}\in |\mathfrak{B}|^m$ or \textit{relatively
interpretable} in $\mathfrak{B}$ if there is a set of first-order formulas
$$\Psi=\{A(\bar{x},\bar{y}), E(\bar{x},\bar{y}_1,\bar{y}_2),\Psi_{\sigma}(\bar{x}, \bar{y}_1, \ldots ,
\bar{y}_{t_{\sigma}}): \sigma \textrm{  a predicate of signature  } \Sigma \}$$ of signature $\Delta$ such
that
 \begin{enumerate}
 \item $A(\bar{b})=\{\bar{a}\in|\mathfrak{B}|^n:\mathfrak{B}\models A(\bar{b},\bar{a})\}$ is not empty,
 \item $E(\bar{x},\ov{y}_1,\ov{y}_2)$ defines an equivalence relation $\epsilon_{\bar{b}}$ on $A(\bar{b})$,
 \item if the equivalence class of a tuple of elements $\bar{a}$ from $A(\bar{b})$ modulo the
 equivalence relation $\epsilon_{\bar{b}}$ is denoted by $[\bar{a}]$, for every $n$-ary predicate
 $\sigma$ of signature $\Sigma$, the predicate $P_{\sigma}$ is defined on
 $A(\bar{b})/\epsilon_{\bar{b}}$ by
 $$P_{\sigma}([\bar{b}],[\ov{a}_1], \ldots, [\ov{a}_n])\Leftrightarrow_{\text{def}}\mathfrak{B}\models \Psi_{\sigma}(\bar{b}, \ov{a}_1,
 \ldots, \ov{a}_n),$$
 \item There exists a map $f:A(\bar{b})\rightarrow |\fra{U}|$ such that the structures $\mathfrak{U}$ and
 $\Psi(\mathfrak{B},\bar{b})=\langle A(\bar{b})/\epsilon_{\bar{b}},P_{\sigma}:\sigma\in \Sigma \rangle$ are isomorphic via the map $\tilde{f}:A(\bar{b})/\epsilon_{\bar{b}}\rightarrow |\fra{U}|$ induced by $f$.
 \end{enumerate}
 Let $\Phi(x_1,\ldots, x_n)$ be a first-order formula of signature $\Delta$. If $\mathfrak{U}$ is interpretable in $\mathfrak{B}$ for any 
  parameters $\bar{b}$ such that $\mathfrak{B}\models \Phi(\bar{b})$ then  $\mathfrak{U}$
 is said to be \textit{regularly interpretable} in $\mathfrak{B}$ with the help of the formula $\Phi$. If the tuple $\bar{b}$ is empty, $\mathfrak{U}$ is said
 be \textit{absolutely interpretable} in $\mathfrak{B}$. 
 
Now let $T$ be a theory of signature $\Delta$. Suppose that $S: Mod (T) \rightarrow K$ is a functor defined on the class $Mod (T)$ of all models of the theory $T$ (a category with isomorphisms) into a certain category $K$ of structures of signature $\Sigma$. If there exists a system of first-order formulas $\Psi$ of signature $\Delta$, which absolutely interprets the system $S(\frak{B})$ in any model $\frak{B}$ of the theory $T$ we say that $S(\frak{B})$ is \textit{absolutely interpretable in $\frak{B}$ uniformly with respect to $T$}.
 
 For example, the center $Z(G)$ of a group $G$ is interpretable (or in this case definable) in $G$ uniformly with respect to the theory of groups. On the other hand, the commutator subgroup $G'$, generally speaking, is not interpretable in $G$ uniformly with respect to the theory of groups. However, it is so if $G$ is a finitely generated nilpotent group.
 
 \subsubsection{$A$-Modules as two-sorted structures} Assume $A$ is a commutative associative ring with unit and $M$ is an $A$-module. For us the $A$-module $M$ is a two-sorted structure $M_A=\langle M, A, s\rangle$ where $A$ is a ring, $M$ is an abelian group and $s=s(x,y,z)$, where $x$ and $z$ range over $M$ and $y$ ranges over $A$ is the predicate describing the action of $A$ on $M$, that is $\langle M,A,s\rangle \models s(m,a,n)$ if and only if $a\cdot m =n$. Sometimes we drop the predicate $s$ from our notation and write $M_A=\langle M, A\rangle$. When we say that \emph{the ring $A$ and its action on $M$ are interpretable in a structure $\mathfrak{U}$} we mean that the one-sorted structure naturally associated to $M_A$ is interpretable in $\mathfrak{U}$. In particular we shall see that an infinite finitely generated nilpotent group $G$ often contains definable normal subgroups, say $N$ and $M$, where $N\leq M$, $M/N$ is infinite abelian and the richer structure $\langle M/N, \Z\rangle$ is interpretable in $G$. In such a case we shall say that the ring of integers $\Z$ and its action on $M/N$ are interpretable in $G$. 
 
Note that if a multi-sorted structure has signature without any function symbols then there is a natural way to associate a one-sorted structure to it. We always assume that our signatures do not contain any function symbols, since functions can be interpreted as relations. Therefore when we talk about interpretability of multi-sorted structures in each other or interpretability of a multi-sorted structure into a one-sorted one we mean the interpretability of the associated one-sorted structures.  

Recall that a \emph{homomorphism $\theta : \langle M,A, s\rangle \to \langle N,B, t\rangle$ of two-sorted modules} is a pair $(\theta_1, \theta_2)$ where $\theta_1: M\to N$ is a homomorphism of abelian groups and $\theta_2:A \to B$ is a homomorphism of rings satisfying 
$$ s(m_1,a,m_2)\Leftrightarrow t(\theta_1(m_1), \theta_2(a), \theta_1(m_2)), \quad \forall a\in A, \forall m_1,m_2\in M.$$  A homomorphism $\theta$ as above is said be \emph{an isomorphism of two-sorted modules} if $\theta_1$ and $\theta_2$ are isomorphisms of the corresponding structures. 
 \subsection{Nilpotent groups}  
 Here we introduce basic definitions from the theory of nilpotent groups. Our basic references for nilpotent groups are \cite{hall} and \cite{war}.

Consider elements $x$ and $y$ of a group $G$. Let $[x,y]=x^{-1}y^{-1}xy$. We call $[x,y]$ the \textit{commutator} of the elements $x$ and $y$.  If $H$ and $K$ are subgroups of $G$, $[H,K]$ is the subgroup of
$G$ generated by commutators $[x,y]$, $x\in H$ and $y\in K$.
 A series of subgroups of $G$:
$$G=G_1\geq G_2 \geq \ldots G_n \geq G_{n+1}\geq \ldots ,$$
is called \emph{central} if $[G,G_i]\leq G_{i+1}$ for each $i\geq 1$. If $G$ has a central series as above and there exists a natural number $n$ such that $G_{n+1}=1$ then we say that $G$ is a \emph{nilpotent} group. 

 Assume $G$ is a nilpotent group. Let us define a series $\Gamma_1(G), \Gamma_2(G),
\ldots$ of subgroups of $G$ by setting
$$G=\Gamma_1(G), \quad \Gamma_{n+1}(G)=[\Gamma_n(G),G]\quad \textrm{for all $n>1$}\index{ $\Gamma_i(G)$}.$$
It can be easily checked that the above series is a central series. If $G$ is clear from the context we write $\Gamma_i$ for $\Gamma_i(G)$. We often denote $\Gamma_2(G)$ by $G'$ and we also use $Ab(G)$ for the abelian quotient $G/G'$. If $c$ is the least number such that
$\Gamma_{c+1}(G)=1$ then $G$ is said to be a \emph{nilpotent group of class $c$} or simply a
\textit{$c$-nilpotent} group. 

Let $Z(G)=\{x\in G: xy=yx, \forall y\in G\}$ denote the \emph{center of the group $G$}. We define a series of subgroups $Z_i(G)$ of $G$ by setting
$$Z_1(G)=Z(G), \quad Z_{i+1}(G)=\{x \in G:xZ_i(G)\in Z(G/Z_i(G))\},\quad i\geq 1.$$ This series is also a central series and called the \textit{upper central
series} of the group $G$. If $G$ is clear from the context we write $Z_i$ from $Z_i(G)$. Note that if $Z_{c}(G)=G$ and $Z_{c-1}(G)\neq G$ is and only if $G$
is a $c$-nilpotent group.
\subsubsection{Nilpotent groups admitting exponents in a binomial domain}
A \emph{binomial domain $A$} is a characteristic zero integral domain such that for all elements $a\in A$ and copies of positive integers $k=\underbrace{1+ \cdots + 1}_{k-\text{times}}$ there exists a (necessarily unique) solution in $R$ to the equation:
 $$a (a-1)\cdots (a -k+1)= x (k!).$$
 Such a solution is denoted by $\binom{a}{k}$.
 As examples of binomial domains one could mention any characteristic zero field, the ring of integers $\mbb{Z}$, or any subring of the field of rationals $\Q$ (See~\cite{war}).  It is easy to check that being a binomial domain is a first-order property. So any model of the first-order theory $Th(\Z)$ of the ring of integers $\Z$ is also a binomial domain. 
 
 A group $G$ \emph{admitting exponents in a binomial domain $A$} is a nilpotent group $G$ together with a function:
 $$G\times A\rightarrow G, \quad (x,a)\mapsto x^a,$$
 satisfying the following axioms:
 \begin{enumerate}
  \item $x^1=x$, $x^a x^b=x^{(a+b)}$, $(x^a)^b=x^{(a b)}$, for all $x\in G$ and $a, b\in A$.
  \item $(y^{-1}xy)^{a}=y^{-1}x^a y$ for all $x,y\in G$ and $a\in A$.
  \item  $x_1^a x_2^a\cdots x_n^a=(x_1x_2\cdots x_n)^a\tau_2(\bar{x})^{\binom{a}{2}}\cdots \tau_c(\bar{x})^{\binom{a}{c}}$, for all $x_1$, \ldots, $x_n$ in $G$, $a \in A$, where the terms $\tau_i$ come from Hall-Petresco formula and $c$ is the nilpotency class of $G$.
 \end{enumerate}
These are also referred to as \emph{$A$-exponential groups} or  \emph{$A$-groups} in the literature. Homomorphisms of $A$-groups, $A$-subgroups, and similar notions are defined in the obvious manner. We say that a group $G$ is a \emph{finitely generated nilpotent $A$-group} if $G$ is a nilpotent $A$-group which is finitely generated as an $A$-group. We refer the reader to either of \cite{hall} or \cite{war} for more details regarding $A$-groups.  

Model-theoretically a nilpotent $A$-group $G$ can be considered as a pure group, i.e. as a structure with language $L$, or as a two-sorted structure $G_{A}=\langle A, G, s\rangle$ similar to two-sorted modules. We denote the corresponding language by $L_1$. Indeed even a finitely generated nilpotent group $G$ can be considered as a two-sorted structure $G_{\Z}$. Again our first aim in this paper is to understand to what extend $G_{\Z}$ can be recovered from $G$.

 \subsection{Bilinear maps} 
 Assume $M_1$, $M_2$ and $N$ are $A$-modules, where $A$ is a commutative associative ring with unit. The map 
 $$f:M_1\times
M_2\rightarrow N$$
is called $A$-bilinear if 
$$f(ax,y)=f(x,ay)=af(x,y)$$ 
for all $x\in M_1$, $y\in M_2$ and $a \in A$.
 An $A$-bilinear mapping $f:M_1\times M_2\rightarrow N$ is called \textit{non-degenerate in the first variable} if $f(x,y)=0$ for all $y$ in $M_2$ implies $x=0$. Non-degeneracy with respect to the second variable is defined similarly. The mapping $f$ is called \textit{non-degenerate} if it is non-degenerate with respect to first and second variables.
We call the bilinear map $f$, a \textit{full bilinear mapping} if $N$ is generated as an $A$-module by elements $f(x,y)$, $x\in M_1$ and $y\in M_2$. 
\subsection{The second cohomology group and extensions of groups}
Here we briefly review a few concepts from cohomology of groups. The aim is to state three well-known facts on the relationship between the second cohomology group and extensions of groups. Readers may refer to (\cite{robin}, Chapter 11) or (\cite{Brown}, Chapter IV) for details.

Assume \begin{equation}
\label{ext1:eqn}1 \rightarrow N\xrightarrow{\mu} E \xrightarrow{\epsilon} G \rightarrow 1,\end{equation}
is a short exact sequence of groups, called an \emph{extension of $N$ by $G$}. The group $N$ is often called the \emph{kernel} of the extension. If all involved groups are abelian then the sequence is called an \emph{abelian extension of $N$ by $G$}. Consider a function $\eta: G \to E$ transversal to $\epsilon$, i.e. $\eta$ satisfies $\epsilon\circ\eta=1_G$. Then $\eta$ gives a function $\chi_{\eta}: G\rightarrow Aut(N)$ induced by the action of $\eta(G)$ on $\mu(N)$ by conjugation in $E$. As usual $Aut(N)$ denotes the group of automorphisms of $N$. Given the data above the result of the action of $x\in G$ on $y\in N$ is denoted by $y^x$.

If $N$ is an abelian group the above action of $G$ on $N$ induces a $\Z[G]$-module structure on $N$, where $\Z[G]$ is the integral group ring of $G$. Given this data we say that $N$ is a \emph{$G$-module}.

Assume that $N$ is a $G$-module and the following is a possibly different extension of $N$ by $G$.
\begin{equation}
\label{ext2:eqn}1\rightarrow N \xrightarrow{\mu'} E' \xrightarrow{\epsilon'} G \rightarrow 1\end{equation}
We say that the extensions \eqref{ext1:eqn} and \eqref{ext2:eqn} are \textit{equivalent} if there exists an isomorphism $\phi:E \rightarrow E'$ of groups such that the diagram:
$$\begin{CD}
1 @>>> N @>{\mu}>> E @>{\epsilon}>> G @>>> 1 \\
@.  @| @VV{\phi}V  @| @. \\
1 @>>> {N} @>>{\mu'}> E' @>>{\epsilon'}> G @>>> 1
\end{CD}$$
commutes. This introduces an equivalence relation on extensions of the $G$-module $N$ by $G$. 

Let $\eta(1_G)=1_E$. A function $f:G\times G \to N$ satisfying:
\begin{itemize}
\item $f(1_G,y)=1_N=f(x,1_G), ~\forall x,y\in G$
\item $f(x,y)f(xy,z)=(f(y,z))^xf(x,yz),~ \forall x,y,z\in G$
\end{itemize} is called a \emph{2-cocycle} from $G$ to $N$ given the action of $G$ on $N$. These 2-cocycles form an abelian group under point-wise multiplication denoted by $Z^2(G,N)$. A 2-cocycle $f$ is called a \emph{2-coboundary} if there is a function $h:G\to N$ satisfying $h(1_G)=1_N$ and 
$$f(x,y)=(h(y))^{x}(h(xy))^{-1}h(x),$$
for all $x,y \in G$. 
2-coboundaries form a subgroup $B^2(G,N)$ of $Z^2(G,N)$. Define \emph{the second cohomology group} given this data by: 
$$H^2(G,N)=\frac{Z^2(G,N)}{B^2(G,N)}.$$
\begin{fact}\label{co-1} If $N$ is a $G$-module then there exists a bijection between the equivalence classes of extensions of $N$ by $G$ realizing the action and the group $H^2(G,N)$. 
\end{fact}
 If the action of $G$ on $N$ is trivial and $G$ is also abelian, then a 2-cocycle $f\in Z^2(G,N)$ is called \emph{symmetric} if $f(x,y)=f(y,x)$ for all $x$ and $y$ in $G$. The subgroup of symmetric 2-cocycles of $Z^2(G,N)$ is denoted by $S^2(G,N)$. We define $Ext(G,N)$ by
 $$Ext(G,N)=\frac{S^2(G,N)}{S^2(G,N)\cap B^2(G,N)}.$$
\begin{fact}\label{co-0} If $N$ is a trivial $G$-module and $G$ is abelian, then there exists a bijection between the equivalence classes of abelian extensions of $N$ by $G$ and the group $Ext(G,N)$.
\end{fact}

 Now assume $N$ is not necessarily abelian. Then the function $\chi_\eta$ induces a homomorphism $\chi: G\rightarrow Out(N)$ called a \emph{coupling} associated to the extension which is independent of $\eta$, where $Out(N)$ denotes the group of outer-automorphisms of $N$.  It is a known fact that there are couplings which can not be realized by any extensions. However given a coupling $\chi$ and knowing that there exists at least one extension realizing it, we may lift $\chi$ to a set map with target $Aut(N)$, then indeed this induces a well-defined homomorphism from $G$ to $Aut(Z(N))$, that is, the center $Z(N)$ of $N$ is actually a $G$-module. 
 
\begin{fact}\label{co} There is a bijection between the equivalence classes of extensions of $N$ by $G$ with coupling $\chi$ and the group $H^2(G,Z(N)),$ providing that such extensions exist, in which case $Z(N)$ is a $G$-module via any lifting of $\chi:G\to Out(N)$ to $Aut(N)$.
\end{fact}
 \subsection{Our approach and statement of the main results}\label{approach:sec}
 
Here we provide an outline of the paper, a lay out of the main arguments and formulate our main results. It is regrettably a long one, even though the authors have tried to keep the details minimal and presentation somehow informal except in certain cases. But hopefully it gives a good outline of the main arguments and results. In later sections of the paper we provide all details, formal description of the constructions, and the proofs. 

Given a finitely generated nilpotent group $G$ our first aim is to construct a central series 
$$G=G_1 > G_2> \ldots > G_m > G_{m+1}=1,$$
where each $G_i$ is first-order definable in $G$ and each quotient $G_i/G_{i+1}$ satisfies one of the following mutually exclusive conditions:
\begin{enumerate}
\item[(a)] $G_i/G_{i+1}$ is finite,
\item[(b)] $G_i/G_{i+1}$ is infinite torsion-free abelian and the two-sorted $\Z$-module $\langle G_i/G_{i+1}, \Z\rangle$ is interpretable in $G$ uniformly with respect to $Th(G)$,
\item[(c)] $G_i/G_{i+1}$ is infinite torsion-free abelian and there is no definable subgroup $K/G_{i+1}$ of $G_i/G_{i+1}$ such that $\langle K/G_{i+1}, \Z\rangle$ is interpretable in $G$.
\end{enumerate}
Let us call such a central series a \emph{maximally refined series}. It may not be very hard to see that if one is able to construct a series for $G$ with quotients satisfying (a) and (b) only, then any finitely generated group $H$ elementarily equivalent to $G$ will be isomorphic to it. That is because we can use the interpretability of the ring of integers and its action or the finiteness of the quotient to express in the first-order theory of groups that each of the quotients $G_i/G_{i+1}$ is generated by certain elements say $u_{ij}G_{i+1}$, $j=1, \ldots, m_i$ for some positive integer $m_i$. This will imply that $G$ is generated by the $u_{ij}$ and it is not hard to see that $G$ being finitely presentable we can also express all relations among the $u_{ij}$ in the first-order theory of groups. Now, all definitions and interpretations being uniform with respect to (at least finitely generated models of) $Th(G)$ implies that $H$ has the same presentation. The existence of quotients of type in (c) and the way the rest of the group interacts with these quotients determines (at-least heuristically) the extent to which the structure of $H$ can deviate from that of $G$. We shall see that the only quotient satisfying (c) in $G$ is $(Is(G')\cdot Z(G)) / Is(G')$ or equivalently $Is(G')/(Is(G')\cap Z(G))$, where $Is(G')$ refers to the isolator of the commutator subgroup $G'$. Note that if $G$ is abelian then one could get only quotients of type (a) and (c). In this case one can use Szmielew's invariants to show that $H$ and $G$ are isomorphic. There is an example, due to Zilber~\cite{Z71}, of a class 2 finitely generated nilpotent group $G$ such that $G$ has a quotient of type (c). The tools developed in this paper shed a new light on this example (which will be studied in the final section).  

So, let us keep constructing a maximally refined series as an ideal goal. With this goal in mind let us describe a few intermediate constructions which will get us there.

Let
$$G=R_1>R_2> \ldots > R_c> R_{c+1}=1,\qquad (R),$$ be an arbitrary central series of the nilpotent group $G$. Let
$$R^{u}_{c+1}=1, \text{and} ~R^u_i=\{x\in G: [x,G]\subseteq [R_i,G] \},\quad 1\leq i\leq c,$$ and
$$ R^l_1=G,\textrm{  and  } R^l_i=[R_{i-1},G], \quad 2\leq i\leq c+1.$$
Each $R^u_i$ and $R^l_i$ is a subgroup of $G$, and
$$R^u_i\geq R_i\geq R^l_i, \quad 1\leq i\leq c,$$
and
$$[R^u_i,G]=[R_i,G]=R^l_{i+1}, \quad 1\leq i\leq c.$$
Hence
$$G=R^u_1> R^u_2> \ldots > R^u_c=Z(G)>R^u_{c+1}=1, \quad (R^u)$$
$$G=R^l_1> R^l_2=G' > R^l_3> \ldots >R^l_{c+1}=1, \quad (R^l)$$
are central series for $G$. We call the series $(R^u)$ and $(R^l)$ the upper and
lower series associated with $(R)$, respectively. The operation of commutation induces well-defined full bilinear maps
$$f_i:G/G'\times R^u_i/R^u_{i+1} \rightarrow R^l_{i+1}/R^l_{i+2}, \quad
1\leq i\leq c-1$$
each of which is  non-degenerate with respect to
the second variable.
The mappings constructed form a bundle $S_R=\{f_1, \ldots f_{c-1} \}$ of
bilinear mappings associated with the series $(R)$.

Let
$$R^u=R^u_1/R^u_2\oplus \ldots \oplus R^u_{c-1}/R^u_c,$$
$$R^l=R_2^l/R_3^l\oplus \ldots \oplus R^l_c/R^l_{c+1}.$$
We call each quotient $R^u_i/R^u_{i+1}$ a \emph{major direct factor} of $R^u$. Similarly the $R^l_{i}/R^l_{i+1}$ are called major direct factors of $R^l$
\begin{prop}
\label{FR:prop} The bundle $S_R$ induces a full non-degenerate with respect to both variables bilinear map
$$F_R:\frac{G}{V_R}\times R^u \rightarrow R^l,$$
for some subgroup $V_R \geq G'$.
\end{prop} 
 The details of the construction and the proof of the proposition can be found in Section~\ref{bilincon}.

Now one can associate a (necessarily commutative associative unitary) ring $P(F_R)$, called the largest ring of scalars of the mapping $F_R$
(see Section~\ref{bilin}) to $F_R$. The ring $P(F_R)$ may not act on all the major direct factors of $R^u$ and $R^l$. But the largest subring $P_R$ of $P(F_R)$ that stabilizes these major direct factors exists and moreover is definable in $P(F_R)$.  So the abelian groups $G/V_R$,
$R^u_i/R^u_{i+1}$ and $R^l_{i+1}/R^l_{i+2}$, $1\leq i \leq c-1$ carry  $P_R$-module structures interpretable in $G$. 

The argument above shows that the ring $P_R$ acts on all the quotients of the upper series except, possibly, the one corresponding to the lowest gap $R^u_c=Z(G)> 1$. It also acts on all the quotients of the lower series except possibly on the quotient of the upper gap $G> G'
=R^l_2$.

 In Section~\ref{coord:sec} we shall introduce refinements $(U(R))$ and $(U(L))$ of $(R^l)$ and $(R^u)$, respectively, that minimize this issue. We shall also construct a definable subring $A_R$ of $P_R$ that acts simultaneously on all quotients of the consecutive terms of the two new series except possibly for the quotient associated to one gap in each case. In the case of $(U(R))$ the gap is $Z(G)> Z(G)\cap G'$ and in the case of $(L(R))$ the gap is $Z(G)\cdot G'>G'$. Notice that if $Z(G)\leq G'$ the gaps are trivial. If they exist they give isomorphic quotients. So we refer to both gaps as the \emph{special gap}.        

The importance of all these constructions for studying the elementary theory of $G$ is reflected in the following proposition. 

\begin{prop}\label{prop64} Assume $A$ is an arbitrary binomial domain. Let $G$ be a finitely generated $A$-exponential nilpotent group  and 
$$G=R_c>R_{c-1}>\ldots R_1>R_0=1 \qquad (R)$$
be a central series of $A$-subgroups where each term is definable in $G$ uniformly with respect to $Th(G)$ . Then,
\begin{enumerate}
\item the upper completed series $(U(R))$  associated with $(R)$ is an $A$-central series and all its terms are interpretable in $G$ uniformly with respect to $Th(G)$,
\item the lower completed series $(L(R))$  associated with $(R)$ is an $A$-central series and all its terms are interpretable in $G$ uniformly with respect to $Th(G)$,
\item the ring $A_R(G)$ is an $A$-algebra, and it is interpretable in $G$ uniformly with respect to $Th(G)$, moreover:
\begin{itemize} 
\item[(a)] the action of $A_R(G)$ on all quotients of the upper completed associated series $(U(R))$ (maybe except the gap $Z(G)\geq G'\cap Z(G))$ is interpretable in $G$ uniformly with respect to $Th(G)$,
\item[(b)] the action of $A_R(G)$ on all quotients of the lower completed associated series $(L(R))$ (maybe except the gap $Z(G)\cdot G' >G'$) is interpretable in $G$ uniformly with respect to $Th(G)$.\end{itemize}\end{enumerate}\end{prop}
The constructions are introduced in Section~\ref{coord:sec}. The logical properties of the constructions are studied in Section~\ref{logical:sec}. 

\begin{rem} As starting point for building a maximally refined series for $G$ one can choose any central series with good logical properties, such as the lower central series or the upper central series for $G$. However the existence of the gaps appearing in 3(a) and 3(b) of Proposition~\ref{prop64} and the isomorphism type of the corresponding quotients are independent of the original series $(R)$. That is why we call it the special gap.\end{rem}       

We would have already achieved the goal of constructing a maximally refined series if we could somehow interpret (define) the ring of integers $\Z$ in the ring $A_R(G)$. This is the topic we discuss in Section~\ref{Z-interpret:sec}. By construction the ring $A_R(G)$ is a subring of the endomorphism ring of $G/V_R(G)$ which is a finitely generated abelian group. Therefore it is an example of a ring $A$ with finitely generated additive group $A^+$. It is also a Noetherian ring. Let us recall that any such ring admits a decomposition of zero
$$0=\p_1\cdot \p_2 \cdots \p_m,$$ 
where the $\p_i$ are certain prime ideals of $A$. 

\begin{prop}\label{informal-coord:prop} Let $A$ be a commutative associative ring with unit and finitely generated additive group $A^+$. Then,
\begin{enumerate}
\item[(a)] there exists a definable decomposition of zero $0=\p_1\cdot \p_2 \cdots \p_m,$
in $A$,
\item[(b)] for each $i$, either the characteristic, char$(A/\p_i)$, of the integral domain $A/\p_i$ is not zero and so $A/\p_i$ is finite or char$(A/\p_i)=0$ and the ring of integers $\Z$ is definable in $A/\p_i$. 
\end{enumerate}
\end{prop} 
Part (a) is actually the statement of Proposition~\ref{primary:thm} and part (b) is a corollary of Lemma~\ref{julia:lem} and part (a). These are informal versions of the more technical statements appearing in  Sections~\ref{Prime-decom:sec} and \ref{elem-eq-rings:sec}. Indeed for what follows we need a stronger statement which is Proposition~\ref{basedef}. Let us just comment on part (b), which uses the celebrated theorem of Julia Robinson on undecidability of algebraic number fields of finite degree over $\Q$ in \cite{julia}. She proves that the ring of integers $\Z$ is interpretable in any algebraic number field, that is, a finite extension of the field of rational numbers $\Q$. Note that if $A$ is a characteristic zero integral domain with finitely generated $A^+$ then the field of fractions $F$ of $A$ is an algebraic number field. It is an easy exercise to show that $F$ is interpretable in $A$ uniformly with respect to $Th(A)$. Now the statement follows from Robinson's theorem.

Now combining Proposition~\ref{prop64} and \ref{informal-coord:prop} modulo certain technicalities one may prove the following proposition. We will provide the proof at the beginning of Section~\ref{main:sec}. 

\begin{prop}\label{max-refined:prop} Assume $G$ is a finitely generated nilpotent group which is not abelian-by-finite. then, at least two of the quotients of consecutive terms $(U(R))$ (respectively $(L(R))$) is infinite. Moreover the ring of integers $\Z$ and its action on all the infinite quotients of the consecutive terms of $(U(R))$ (respectively ($L(R)$)), except possibly $Z(G)/(Z(G)\cap G')$ (respectively $(Z(G)\cdot G')/G'$), are interpretable in $G$.  All the interpretations are uniform with respect to finitely generated models of $Th(G)$.  \end{prop}

It should not come as surprise that we also obtain a result on elementary equivalence of finitely generated modules over commutative associative rings with unit and finitely generated additive group. We recall that for us an $A$-module $M$ is a two sorted structure $\langle M, A, s \rangle$, where $M$ is  an abelian group, $A$ is a ring and $s$ is the predicate describing the action of $A$ on $M$. Denote the language by $L_2$. 

 \begin{thm}\label{elemmod:thm}Let $A$ be an associative commutative ring with unit and let $M$ be a finitely generated $A$-module. Then
 there exists a sentence $\psi_{M,A}$ of the language $L_2$ such that $\langle M,A\rangle\models \psi_{M,A}$ and for
 any ring $B$ with finitely generated $B^+$ and any finitely generated $B$-module $N$, we
 have
 $$\langle N,B\rangle\models \psi_{M,A} \Leftrightarrow \langle N,B\rangle
 \cong \langle M,A\rangle.$$\end{thm} 
 
 The proof of the theorem appears at the end of Section~\ref{Z-interpret:sec}. Indeed Theorem~\ref{elemmod:thm} implies the next two statements.  
 \begin{cor} For any associative commutative ring $A$ with unit and
finitely generated additive group $A^+$ there exists formula $\psi_A$
such that $A\models \psi_A$ and for any ring $B$ with finitely generated $B^+$ we have
$$B\models \psi_A \Leftrightarrow A\cong B.$$\end{cor}
\begin{cor}Let $\mathcal{K}$ be the class of all associative
commutative rings with finitely generated additive group. Then any $A$
from $\mathcal{K}$ is finitely axiomatizable inside
$\mathcal{K}$.\end{cor}
Let us denote by $L_3$ the first-order language of two-sorted algebras. An algebra $\langle  C, A\rangle$ consists of an arbitrary ring $C$, not necessarily commutative or associative or unitary and the scalar ring $A$ which is assumed to be commutative associative and unitary The following result follows from the statements above. We provide the proof right after the proof of Theorem~\ref{elemmod:thm} in Section~\ref{Z-interpret:sec}.  
\begin{thm}\label{elem-iso-alg:thm} Let $\mathcal{A}$ be the class of all two-sorted algebras $\langle C, A\rangle$ where $C$ is finitely generated as an $A$-module and $A^+$ is finitely generated as an abelian group. For each  $\langle C, A\rangle\in \mathcal{A}$ there exists a formula $\phi_{C,A}$ of $L_3$ such that $\langle C, A\rangle\models \phi_{C,A}$ and for any $\langle D, B\rangle \in \mathcal{A}$,
 $$\langle D, B\rangle\models \phi_{C,A} \Leftrightarrow \langle C,A\rangle \cong \langle D,B\rangle$$
 as two-sorted algebras.    \end{thm}

Going back to the discussion of nilpotent groups, in some sense Proposition~\ref{max-refined:prop} finishes constructing maximally refined series. In the remaining parts of the paper we shall look at the consequences of having such a series and why it is in general the best we can achieve.

 In Section~\ref{coordinatization} we present a coordinatization theorem for finitely $K$-generated groups taking exponents in a binomial principal ideal domain $K$. Recall that elements in a finitely generated torsion-free nilpotent group $G$ can be represented as tuples in $\Z^m$, where $m$ is the Hirsch number of $G$. By a classical theorem of P. Hall~\cite{hall} product and exponentiation are computed by certain polynomials with rational coefficients and integer values on integer entries. Moreover these polynomials are uniquely determined by the so-called structure constants. Here we do not assume that the groups are torsion-free. So we need more complicated but basic tools and language to describe such a group as a tuple in $K^m$ where its isomorphism type is determined by structure constants and periods of (relative) torsion elements. 
 
Section~\ref{regular:sec} introduces terminology related to the ``special gap'' and also a class of finitely generated nilpotent groups, here called regular groups, where elementary equivalence implies isomorphism. In Section~\ref{regular:sec} we actually deal with a somewhat bigger class of groups (taking exponents in a binomial PID), but for now we just restrict our attention to finitely generated nilpotent groups. So let $G$ be such a group. If $N$ is a subgroup of $G$, define
 $$Is(N)=\{x\in G:\exists n \in \N^*~ (x^n\in N)\}.$$
The set $Is(N)$ is a subgroup of $G$ called the \emph{isolator of $N$} in $G$, which is normal in $G$ if $N$ is a normal subgroup of $G$. Next Define
 $$I(G)=Is(G')\cap Z(G).$$
 The special gap
 \begin{equation*}Z(G)\geq G' \cap Z(G)\end{equation*}
 will be called \textit{tame}, if $Z(G)=I(G)$.
 
 Let us refine the special gap above with the help of $I(G)$:
 $$Z(G)\geq I(G)\geq G'\cap Z(G).$$
 The quotient $I(G)/(G'\cap Z(G))$ is a finite abelian group, and the quotient $Z(G)/I(G)$ is a free abelian group of finite rank. Indeed $Z(G)$ splits over $I(G)$ and therefore there exists a free abelian group of finite rank $G_0\leq Z(G)$, such that $Z(G)=
 G_0\times I(G)$.
 
 Any subgroup $G_0\leq G$ such that
 $Z(G)=G_0\times I(G)$ is called an \emph{addition of $G$} and the quotient group $G_f=G/G_0$
 is called a \emph{foundation of $G$}, associated with the addition
 $G_0$. If $G\cong G/G_f\times G_0$ or equivalently if $Is(G'\cdot Z(G))=Is(G')\cdot Z(G)$ then we call $G$ a \emph{regular} group. The fact that the two conditions given above are equivalent is part of the statement of Proposition~\ref{rproperties:prop}. 

Now let us describe the main results of Section~\ref{main:sec}. This section reflects the meticulous study of what having a ``maximally refined central series'' for a finitely generated nilpotent group $G$ means in understanding the structure of a finitely generated group $H$ elementarily equivalent to it. The main auxiliary results in the section are Lemma~\ref{mainlem} and Proposition~\ref{maincor}. The statements are technical and long, so we do not state them here. Some of the main consequences are summarized in the following theorem. 
\begin{thm}\label{prop:embed} Assume $G\equiv H$ are finitely generated nilpotent groups. Then
\begin{enumerate}
\item[(a)] there exists a monomorphism $\phi:G\to H$ of groups with $im(\phi)$ a finite index subgroup of $H$,
\item[(b)] for any addition $G_0$ of $G$ there exists an addition $H_0$ of $H$ so that $G_0\cong H_0$ and $G/G_0\cong H/H_0$,
\item[(c)] $Is(G')\cong Is(H')$
\item[(d)] $G/Is(G') \cong H/Is(H')$
\item[(e)] $Is(G'\cdot Z(G))/(Is(G')\cdot Z(G)) \cong Is(H'\cdot Z(H))/(Is(H')\cdot Z(H))$.
\end{enumerate}
\end{thm} 

The most immediate consequence of Theorem~\ref{prop:embed} is the following theorem on elementary equivalence of regular groups. This theorem was announced in \cite{M1984} but the proof was not published.  

\begin{thm}[\cite{M1984}]\label{regular:prop} If $G$ is finitely generated regular group and $H$ is a finitely generated group such that $G\equiv H$, then $G\cong H$.\end{thm}

In Section~\ref{elemcohom:sec} we give a cohomological account of the results in Section~\ref{main:sec}. Let us set $N(G)=Is(G')\cdot Z(G)=Is(G')\times G_0$, $M(G)=Is(G'\cdot Z(G))$ and $\bar{G}= G/N(G)$. We look at $G$ as the following extension where $\mu$ is the inclusion and  $\pi$ is the canonical surjection.
\begin{equation*} 1\rightarrow N(G) \xrightarrow{\mu} G \xrightarrow{\pi} \bar{G} \rightarrow 1.\end{equation*}
Recall that all the equivalence classes of the extensions of $N(G)$ by $\Bar{G}$ with the same coupling, $\chi: \bar{G}\to Out(N(G))$, as this extension are in one-one correspondence with elements of the second cohomology group $H^2(\bar{G}, Z(N(G)))$ where the center $Z(N(G))$ of $N(G)$ is a $G$-module via any lifting of $\chi$ to $Aut(N(G))$. We notice that picking an addition $G_0$ for $G$, the properties $N(G)=G_0 \times Is(G')$ and  $G_0\subset Z(G)$ imply that $Z(N(G))=G_0\times N_1(G)$ where both $G_0$ and $N_1(G)=Z(Is(G'))$ are clearly stable under the action of $G$. We will argue that
 \begin{equation*}\begin{split}
 H^2(\bar{G}, Z(N(G)))&\cong H^2(\bar{G}, N_1(G))\oplus Ext(\frac{M(G)}{N(G)}, G_0)\\
    [f]&\mapsto [f_1] \oplus [f_2],\end{split}\end{equation*}
where $f$, $f_1$ and $f_2$ are 2-cocycles and $[f]$ denotes the class of $f$ in the corresponding cohomology group. This splitting as well as the main results from Section~\ref{main:sec} imply that any finitely generated group $H$ elementarily equivalent to $G$ is realized as an extension of $N(G)$ by $\bar{G}$ with the same coupling as the extension above. If the 2-cocycle corresponding to $G$ (as a representative of an equivalence class) has the form  $f_1+f_2$ then the one corresponding to $H$ has the form $f_1+f_2'$, where the symmetric 2-cocycles $f_2$ and $f_2'$ belong to possibly different classes of $Ext(M(G)/N(G), G_0).$ We can actually say more about the 2-cocycle $f_2'$.

We will define now what an abelian deformation of a finitely generated nilpotent group is. It might help the reader to at least read the statements of Lemmas~\ref{sp1},~\ref{mainlem} and~\ref{stb} before going through the definition. The indices $i_0,i_1$ and $i_2$ used in the definition are defined in Lemma~\ref{sp1}, even though it should be clear from the definition what they are.    
 
  \begin{defn}[Cohomological definition of finitely generated abelian deformations] \label{abdefcohom:defn} Assume $G$ is a finitely generated nilpotent group corresponding to an element  
 $$[f_1]\oplus [f_2] \in H^2(\bar{G}, N_1(G))\oplus Ext(\frac{M(G)}{N(G)}, G_0)\cong H^2(\bar{G}, Z(N(G))),$$
 realizing a coupling $\chi: G\to Out(N(G)).$ Assume 
 $$\frac{M(G)}{N(G)}=\langle u_{i_0+1}N(G)\rangle \times\cdots\times \langle u_{i_1}N(G)\rangle\cong \frac{\Z^+}{e_{i_0+1}\Z^+}\oplus\cdots\oplus \frac{\Z^+}{e_{i_1}\Z^+},$$
 is the invariant factor decomposition of the finite abelian group $M(G)/N(G)$. Set $e=e_{i_0+1}\cdots e_{i_1}$ and $n=i_1-i_0$. Moreover let $\{u_{i_1+1},\ldots ,u_{i_1+n}, \ldots, u_{i_2}\}$ be a basis for the addition $G_0$ as a free abelian group. Let also $p=i_2-i_1$. Assume we are given integers $d_{i_1+1}, \ldots , d_{i_1+n}$ and an $n\times n$ matrix  $(c_{ij})$ of integers such that 
  \begin{enumerate}
 \item[(a)] gcd$(d,e)=1$ where $d=d_{i_1+1}\cdots d_{i_1+n}$
 \item[(b)] $det(c_{ij})=\pm 1$.
 \end{enumerate}
 Define the group Abdef$(G,\bar{d},\bar{c})$ as the extension of $N(G)$ by $\bar{G}$ with the same coupling $\chi:G\to Out(N(G))$ defining $G$, corresponding to $[f_1]\oplus[f'_2]$, where 
  $$f'_2=\sum_{i=i_0+1}^{i_1} f'_{2i},\qquad  f'_{2i}\in S^2(\langle u_iN(G)\rangle , G_0),$$
 and for $0\leq s,t<e_i$,
 \begin{equation*}
 f'_{2i}(u_i^{s}N(G), u_i^tN(G)) = \left\{
 \begin{array}{lr}
 1 & \text{if }~~ s+t<e_i\\
 \prod_{k=i_1+1}^{i_1+n} u_k^{d_kc_{ik}} & \text{   if }~~ s+t\geq e_i.
 \end{array} \right.
 \end{equation*}
 Any group $H$ isomorphic to Abdef$(G,\bar{d},\bar{c})$ is called an abelian deformation of $G$.
 \end{defn}  

Here is the main theorem of this paper.

\begin{thm}[Characterization Theorem] \label{mainthm}
Assume $G$ is a finitely generated nilpotent group and $H$ is a finitely generated group such that $G \equiv H$. Then there exist tuples of integers $\bar{d}$ and $\bar{c}$, satisfying the conditions in Definition~\ref{abdefcohom:defn} such that
$$ H\cong \text{Abdef}(G,\bar{d},\bar{c}).$$  
\end{thm}
The proof of the theorem is included in Section~\ref{elemcohom:sec}. The converse is proven in Section~\ref{converse:sec}.

Finally in Section~\ref{Zilber:sec} we analyze an example given by B. Zilber~\cite{Z71} of two finitely generated 2-nilpotent groups which are elementarily equivalent but not isomorphic. We will show that they can be looked as abelian deformations of one another.
 
\section{Bilinearization of nilpotent groups}\label{bilinsec}
In this section we define the non-degenerate bilinear mapping $F_R$,
which is canonically associated with an arbitrary
central series $(R)$
 of a nilpotent group $G$. In a certain sense, the mapping $F_R$ is a
generalization of the graded Lie
ring associated with the central series $(R)$ of $G$. 

All the results and definitions in Subsection \ref{bilin} and \ref{largest-ring:sec} are taken from \cite{alexei86}. The rest of the material in this section is new to the best of our knowledge.
\subsection{Largest ring of a bilinear map} \label{bilin}
In this section all the modules are considered to be exact and scalar rings are always commutative associative with a unit. An $A$-module $M$ is said to be \emph{exact} if
 $am=0$ for $a\in A$ and all $m\in M$ implies $a=0$. 
Let $f:M_1\times M_2\rightarrow N$ be a non-degenerate full $A$-bilinear mapping for some ring $A$.

Let $M$ be an $A$-module and let $\mu:A\rightarrow P$ be an inclusion of rings. Then the $P$-module $M$ is an
\textit{$P$-enrichment} of the $A$-module $M$ with respect to $\mu$ if for every $a\in A$ and $m \in M$,
$am=\mu(a)m$. Let us denote the set of all $A$ endomorphisms of the $A$-module $M$ by $End_A(M)$. Suppose the
$A$-module $M$ admits a $P$-enrichment with respect to the inclusion of rings $\mu:A\rightarrow P$. Then every
$\alpha\in P$ induces an $A$-endomorphism, $\phi_{\alpha}:M\rightarrow M$ of modules defined by
$\phi_{\alpha}(m)=\alpha m$ for $m \in M$. This in turn induces an injection $\phi_P:P\rightarrow End_A(M)$ of
rings. Thus we associate a subring of the ring $End_A(M)$ to every ring $P$ with respect to which there is an
enrichment of the $A$-module $M$.

\begin{defn}Let $f:M_1\times M_2\rightarrow N$ be a full $A$-bilinear
mapping and $\mu:A\rightarrow P$ be an inclusion of rings. The mapping $f$ admits $P$-enrichment with
respect to $\mu$ if the $A$-modules $M_1$, $M_2$ and $N$ admit $P$ enrichments with respect to $\mu$ and $f$ remains
bilinear with respect to $P$. We denote such an enrichment by $E(f,P)$.\end{defn} 

We define an ordering $\leq$
on the set of enrichments of $f$ by letting $E(f,P_1)\leq E(f,P_2)$ if and only if $f$ as an $P_1$ bilinear
mapping admits a $P_2$ enrichment with respect to inclusion of rings $P_1\rightarrow P_2$. The largest
enrichment $E_H(f,P(f))$ is defined in the obvious way. We shall prove existence of such an enrichment for a
large class of bilinear mappings. 

The following proposition and its proof are taken from~\cite{alexei86}. We provide the proof here since the ring $P(f)$, whose construction is revealed in the proof, shall be used in constructing various other rings introduced later in the paper.  
\begin{prop}[\cite{alexei86}, Theorem 1] \label{P(f)}If $f:M\times M \rightarrow N$ is a non-degenerate
full $A$-bilinear mapping over a commutative associative ring $A$ with unit, then $f$ admits the largest enrichment.\end{prop}

\begin{proof}
An $A$-endomorphism $\alpha$ of the $A$-module $M$ is called \textit{symmetric} if $$f(\alpha x,y)=f(x,\alpha y)$$ for every $x,y
\in M$. Let us denote the set of all such endomorphisms by $Sym_f(M)$, i.e.

$$Sym_f(M)=\{\alpha\in End_A(M):f(\alpha x,y)=(x,\alpha y),\quad \forall x,y
\in M \}.$$

 Set $$Z=\{\beta \in Sym_f(M):\alpha \circ \beta= \beta\circ\alpha,\quad \forall \alpha \in Sym_f(M)\}.$$ Then $Z$ is non-empty since the unit 1 belongs to $Z$ and it is actually an $A$-subalgebra of $End_A(M)$. For each $n$, let $Z_n$
be the set of all endomorphisms $\alpha$ in $Z$ that satisfy the formula
\begin{equation*}\begin{split}
S_n(\alpha)\Leftrightarrow & \forall \bar{x},\bar{y},\bar{u},\bar{v}
\left(\sum_{i=1}^nf(x_i,y_i)=\sum_{i=1}^{n}f(u_i,v_i) \rightarrow\right. \\
& \left.\sum_{i=1}^nf(\alpha x_i,y_i)=\sum_{i=1}^{n}f(\alpha u_i,v_i) \right).
\end{split}\end{equation*}
i.e.
\begin{equation*}
Z_n=\{\alpha\in Z: S_n(\alpha)\}.\label{zn}
\end{equation*}
Each $Z_n$ is also  an $R$-subalgebra of $Z$. Now set $$P(f)=\cap_{i=1}^\infty Z_n.$$ The identity mapping is in
every $Z_n$ so $P(f)$ is not empty. Since the mapping $f$ is full, for every $x\in N$ there are $x_i$ and
$y_i$, in $M$ such that $x=\sum_{i=1}^nf(x_i,y_i)$ for some $n$. The $P(f)$-module $M$ is exact by construction. Now we can define the action
of $P(f)$ on $N$ by setting $\alpha x=\sum_{i=1 }^nf(\alpha x_i,y_i)$. The action is clearly well-defined since $\alpha$
satisfies all the $S_n(\alpha)$ and makes $N$ into a $P(f)$-module. Moreover for any $x,y\in M$ and $A\in P(f)$ we have $$f(\alpha x,y)=f(x,\alpha y)=\alpha f(x,y),$$
that is, $f$ is $P(f)$-bilinear.

In order to prove that the ring $P(f)$ is the largest ring of scalars, we prove that for any ring $P$ with respect to which $f$ is bilinear, $\phi_P(P) \subseteq P(f)$. Since $f$ is $P$-bilinear $\phi_P(P) \subseteq
Sym_f(M)$. Let $p \in P$ then for $\alpha\in Sym_f(M)$ and $x,y \in M$,
\begin{equation*}\label{myrem}\begin{split}
f(\alpha \circ\phi_{p}(x),y)&=f(\phi_{p}(x),\alpha y)\\
&=\alpha f(x,\alpha y)=\alpha f(\alpha x,y)\\
&= f(\phi_{p}\circ \alpha (x),y).
\end{split}\end{equation*}
Non-degeneracy of $f$ implies that $\phi_{p}\circ \alpha=\alpha\circ \phi_{p}$. Therefore $\phi_P(P)\subseteq
Z$. It is clear that $\phi_{p}$ belongs to every $Z_n$ by bilinearity of $f$ with respect to $P$. Therefore
$\phi_P(P)\subseteq P(f)$, hence $E(f,P)\leq E(f,P(f))$.
\end{proof}

\subsection{Largest ring of scalars as a logical invariant}\label{largest-ring:sec}
Indeed the ring $P(f)$ is interpretable in the bilinear map $f$ providing that $f$ satisfies certain conditions in addition to the ones in Proposition~\ref{P(f)}. 

The mapping $f$ is said to have \textit{finite width}\index{bilinear mapping!of finite width} if there is a natural number $s$ such that for every $u\in
N$ there are $x_i\in M_1$ and $y_i\in M_2$ such that
$$u=\sum_{i=1}^sf(x_i,y_i).$$
The least such number, $w(f)$\index{ $w(f)$}, is the \textit{width} of $f$.

A set $E_1=\{e_1,\ldots e_n\}$ is a \textit{left complete system} for a non-degenerate mapping $f$ if $f(E_1,y)=0$ implies $y=0$. The
cardinality of a minimal left complete system for $f$ is denoted by $c_1(f)$. A right complete system and the number $c_2(f)$ are defined correspondingly.

The \textit{type} of a bilinear mapping $f$, denoted by $\tau(f)$ \index{ $\tau(f)$}, is the triple $$(w(f),c_1(f), c_2(f)).$$ The mapping $f$ is
said to be of finite type if  $w(f)$, $c_1(f)$ and $c_2(f)$ all exist. If $f,g:M_1\times M_2 \rightarrow N$ are bilinear maps of finite type we say that the type of $g$ is less than the type of $f$ and write $\tau(g)\leq \tau(f)$ if $w(g)\leq w(f)$, $c_1(g)\leq c_1(f)$ and $c_2(g)\leq c_2(f)$.

Let $A$ be a commutative ring with unit. Assume $M_1$, $M_2$ and $N$ are exact $A$-modules. Let $f:M_1\times M_2\rightarrow N$ be a $A$-bilinear map. We associate two structures to $f$. The first one is
$$\mathfrak{U}(f)=\langle M_1,M_2,N, \delta\rangle.$$
where $M_1$, $M_2$ and $N$ are abelian groups and $\delta$ describes the bilinear map. The other one is  $$\mathfrak{U}_A(f)=\langle A,M_1,M_2,N,\delta,s_{M_1},s_{M_2}, s_N\rangle,$$ where $A$ is a ring and $s_{M_1}$, $s_{M_2}$ and $s_N$ describe the actions of $A$ on the modules $M_1$, $M_2$
and $N$ respectively.

We state the following theorem without proof. Readers may refer to the cited reference for a proof. 

\begin{thm}[\cite{alexei86}, Theorem 2]\label{ringinter}Let $f$ be a non-degenerate full bilinear mapping of finite type and let $P(f)$ be the largest ring of scalars of $f$. Then all the constructions, $Sym_f(M)$, $Z$, and $Z_n$ for all $n\in \N^*$ are regularly interpretable in $\mathfrak{U}(f)$. In particular $P(f)=\cap_{n=1}^{w(f)}Z_n$ where $w(f)$ is the width of $f$ and therefore $\mathfrak{U}_{P(f)}(f)$
 is absolutely interpretable in $\mathfrak{U}(f)$. Moreover the formulas involved in the interpretation depend only on the type of $f$.\end{thm}
 
\begin{rem}\label{M1notM2} Proposition~\ref{P(f)} and Theorem~\ref{ringinter} are stated for bilinear maps $f:M_1\times M_2\to N$ where $M_1=M_2$. In the following we reduce the general case $M_1\neq M_2$ to the case in these statements. So assume $f$ is full and non-degenerate. Now define 
$$f^p: (M_1\oplus M_2) \times (M_1\oplus M_2) \to N\oplus N$$
by 
$$f^p(x_1\oplus y_1, x_2 \oplus y_2)=f(x_1,y_2) \oplus f(x_2,y_1),$$
for all $x_i \in M_1$ and $y_i\in M_2$, $i=1,2$. The map is clearly non-degenerate and full. So by Proposition~\ref{P(f)} the largest ring $P=P(f^p)$ of scalars exists. Define $P_1$ as the subset of all elements of $P$ stabilizing both $M_1$ and $M_2$. Indeed $P_1$ is a non-empty subring of $P$ and makes $f$ a $P_1$-bilinear map.  On the other hand any ring $R$ making $f$ bilinear has to embed into $P_1$ otherwise $P(f^p)$ would not define the largest enrichment. So indeed $P(f)=P_1$.

Moreover $f^p$ is absolutely interpretable in $f$. Indeed in this case $M_1$ and $M_2$ are absolutely definable subgroups of $M_1\oplus M_2$ and so the ring $P_1$ is absolutely definable in $P(f^p)$. Therefore $P_1$ is interpretable in $f$. \end{rem}
 \begin{rem} If one only cares about algebraic properties of $P(f)$ there is a simpler definition of $P(f)$. Given a full nondegenerate $R$-bilinear map 
 $$f:M_1\times M_2\to N,$$ one can identify $P(f)$\index{ $P(f)$} with the subring $S\leq End_R(M_1)\times End_R(M_2)\times End_R(N)$ consisting of all triples  $A=(\phi_1,\phi_2,\phi_0)$ such that for all $x\in M_1$ and $y\in M_2$ 
  \begin{equation}\label{myrem1}f(\phi_1(x), y)=f(x,\phi_2(y))=\phi_0(f(x,y)).\end{equation}
  Moreover one could drop the reference to $R$ and simply take $P(f)$ to be the subring $S'$ of $End(M_1)\times End(M_2)\times End(N)$ whose elements satisfy~\eqref{myrem1}. This is simply because $E(f,P(f))\leq E(f,S')$, while $E(f,P(f))$ is the largest enrichment.\end{rem}
 
\subsection{Construction of $F_R$}\label{bilincon}
Throughout $G$ is a nilpotent group. 

Let
$$G=R_1>R_2> \ldots > R_c> R_{c+1}=1,\qquad (R),$$ be an arbitrary central series of the nilpotent group $G$. Let
$$R^u_{c+1}=1, \text{ and } R^u_i=\{x\in G: [x,G]\subseteq [R_i,G] \} \text{ if }  1\leq i\leq c,$$ and
$$ R^l_1=G,\textrm{ and } R^l_i=[R_{i-1},G] \text{ if } 2\leq i\leq c+1.$$
Clearly each $R^u_i$ and $R^l_i$ is a subgroup of $G$, and
$$R^u_i\geq R_i\geq R^l_i, \quad 1\leq i\leq c+1,$$
and
$$[R^u_i,G]=[R_i,G]=R^l_{i+1}, \quad 1\leq i\leq c.$$
Hence
$$G=R^u_1> R^u_2> \ldots > R^u_c=Z(G)> R^u_{c+1}=1 \quad (R^u),$$
$$G=R^l_1> R^l_2=G' > R^l_3> \ldots >R^l_{c+1}=1 \quad (R^l),$$
are central series of $G$.
\begin{defn}The series $(R^u)$ is called the upper central series associated to $(R)$ and $(R^l)$ is called the lower central series associated to $(R)$.\end{defn}
For each $1\leq i \leq c-1$ we may define a full bilinear mapping,
$$f_i:G/G'\times R^u_i/R^u_{i+1} \rightarrow R^l_{i+1}/R^l_{i+2},$$
by $f_i(xG',yR^u_{i+1})=[x,y]R^l_{i+2}$. Each $f_i$, $1\leq i\leq c-1$ is well defined due to the inclusions:
\begin{equation*}\begin{split}
[G',R^u_i]&=[[G,G],R^u_i]\subseteq [[R^u_i,G],G]\\
&=[[G,R^u_i],G]=[[R_i,G],G]\subseteq [R_{i+1},G]\\
&= R_{i+2}^l.\end{split}\end{equation*}
Note that $f_i$ is non-degenerate with respect to the second variable due to the definition of $R^u_{i+1}$. The
mappings constructed form a bundle $S_R=\{f_1, \ldots f_{c-1} \}$ of
bilinear mappings associated with the series $(R)$.

Let
$$V_i(R)=\{x\in G: [x,R^u_i]\subseteq R^l_{i+2}\}$$ be the kernel of $f_i$
with respect to the first variable. Let
$$V_R=\bigcap_{i=1}^{c-1}V_i(R)\geq  G'\cdot  Z(G),$$
$$R^u=R^u_1/R^u_2\oplus \ldots \oplus R^u_{c-1}/R^u_c,$$
$$R^l=R_2^l/R_3^l\oplus \ldots \oplus R^l_c/R^l_{c+1}.$$
We now introduce a full non-degenerate with respect to both variables
 bilinear mapping
$$F_R:G/V_R\times R^u \rightarrow R^l,$$
which canonically corresponds to the bundle $S_R$ and is defined
according to the rule:
$$F_R(x, \sum_{i=1}^{c-1}x_i)=\sum_{i=1}^{c-1}f_i(x_0,x_i),$$
where $x\in G/V_R$, $x_i\in R^u_i/R^u_{i+1}$ and  $x_0$ is an arbitrary
pre-image of $x$ in $G/G'$.

The full non-degenerate bilinear map $F_R$ constructed above will be referred to as \emph{the bilinear mapping associated with the series $(R)$}.

By Theorem~\ref{P(f)} the largest ring of scalars $P(F_R)$ of the bilinear map $F_R$ exists. By construction $P(F_R)$ acts on $G/V_R$, $R^u$ and $R^l$. But it does not necessarily stabilize the major direct factors $R^u_i/R^u_{i+1}$ and $R^l_{i+1}/R^l_{i+2}$, for $1\leq i\leq c-1$. However it is not hard to see that the largest subring of $P(F_R)$ that stabilizes all such quotients is non-empty since it contains the unit 1. So we define $P_R$ as \emph{the largest subring of $P(F_R)$ which stabilizes all $R^u_i/R^u_{i+1}$ and $R^l_{i+1}/R^l_{i+2}$, for $1\leq i\leq c-1$}. 

Let us consider some examples.

\begin{exmp} Consider the lower central series
$$G=\Gamma_1(G)> \ldots > \Gamma_{c}(G)>\Gamma_{c+1}(G)=1,\qquad (\Gamma)$$ of a nilpotent group
$G$. Then we have $\Gamma^l_i(G)=\Gamma_i(G)$. So the action of the ring
$P_{\Gamma}$ on all the quotients
$$\Gamma_2(G)/\Gamma_3(G), \ldots, \Gamma_c(G)/1=\Gamma_c(G)$$
of the lower central series is defined. Note that $G/\Gamma_2(G)$ is excluded from the list above.\end{exmp}
\begin{exmp}For the upper central series
$$G=Z_c(G)> \ldots > Z_1(G)=Z(G)> 1, \qquad (Z)$$
we have $Z^u_i(G)=Z_i(G)$, and the action of the ring $P_Z$ on the quotients
$$Z_c/Z_{c-1}, \ldots , Z_2/Z_1$$
is defined i.e. on all the quotients of the upper central series except the
center $Z(G)$.\end{exmp}
\begin{exmp}\label{2nil}Let $G$ be 2-nilpotent.
\begin{enumerate}
\item From the upper central series
$$G=Z_2>Z_1=Z(G)>1$$
we obtain $Z^u_i=Z_i$, $Z_1^l=G'$, $V_Z=Z(G)$. Therefore $F_Z$ is the
standard bilinear mapping,
$$F_Z:G/Z(G)\times G/Z(G)\rightarrow G'.$$
\item For the lower central series
$$G=\Gamma_1(G)>\Gamma_2(G)=G'>1$$
we obtain $\Gamma^l_i(G)=\Gamma_i(G)$, $\Gamma^u_2(G)=Z(G)$ and
$V_{\Gamma}=Z(G)$, hence
$$F_{\Gamma}:G/Z(G)\times G/Z(G)\rightarrow G'$$
and $F_{\Gamma}=F_Z$.\end{enumerate}
Note that in both cases above the action of the ring
$P_{\Gamma}=P_Z$ is defined on all the quotients except the quotient
$Z(G)/G'$ (if it is non-trivial). \end{exmp}
\begin{lem}\label{Rulbinom:lem} Assume $A$ is a binomial domain, $G$ is a nilpotent $A$-group and the central series $(R)$ is an $A$-series. Then both of the series $(R^l)$ and $(R^u)$ are central $A$-series. Moreover $P_R$ is an $A$-algebra.\end{lem}
\begin{proof}
We first prove that $R^l_i$ are $A$-subgroups. Firstly by Lemma 10.5 of~\cite{war} the commutator map 
$$ R_{c-1}\times G  \to R_c, \qquad(x,y)\mapsto  [x,y]$$
is an $A$-homomorphism in both variables since $R_c\leq Z(G)$. This shows that $R^l_{c}=[R_{c-1},G]$ is an $A$-subgroup. Now fix $1\leq i\leq c-1$ and assume for all $i+1\leq j \leq c$, $R^l_{j}$ is an $A$-subgroup. Now consider the commutator map:
$$ \frac{R_{i-1}}{R^l_{i+1}}\times \frac{G}{R^l_{i+1}}  \to \frac{R_i}{R^l_{i+1}}, \qquad(xR^l_{i+1},yR^l_{i+1})\mapsto  [x,y]R^l_{i+1}.$$
Note that by Lemma~10.6 of \cite{war} and the induction hypothesis all the groups involved in the definition of the map are $A$-subgroups of $G/R^l_{i+1}$. Apply the argument above to the image $R^l_{i}/R^l_{i+1}$ of the map to show that it is an $A$-subgroup of $G/R^l_{i+1}$. This implies that $R^l_{i}$ is an $A$-subgroup by induction. 

To prove that the $R^u_{i}$ are $A$-subgroups we just need to note that 
$$\frac{R^u_{i}}{R^l_{i+1}}=Z\left(\frac{G}{R^l_{i+1}}\right).$$ 
Since the center of an $A$-group is an $A$-subgroup and we proved above that all the $R^l_{i}$ are $A$-subgroups the result follows by induction. 

It remains to prove that $V(R)$ is an $A$-subgroup. By definition $V_i(R)$ is the kernel of the bilinear map $f_i$. By above $f_i$ is $A$-bilinear and $V_i(R)/G'$ is the kernel with respect to the first variable. So for each $V_i(R)/G'$ is an $A$-group, therefore $V_i(R)$ is an $A$-subgroup since $G'$ is so. Hence $V(R)=\cap_{i=1}^{c-1}V_i(R)$ is an $A$-subgroup. In particular this shows that $G/V(R)$ is an $A$-group.     

Indeed we proved above that $F_R$ is an $A$-bilinear map.  This proves that both $P(F_R)$ and $P_R$ are $A$-algebras since by definition they include $A$ as a subring.
 	\end{proof}
\begin{rem} The ring $P_R$ may not be a binomial domain even if $G$ is a nilpotent $A$-group. For example consider the group $G=UT_3(\Z\times \Z)$ of $3\times 3$ upper unitriangular matrices over the ring $\Z\times \Z$. It is easy to verify that $G$ is a 2-nilpotent group with $\Gamma_2=Z_1$. Here $P_\Gamma=P_Z=\Z\times \Z$ while $\Z\times \Z$ is not a domain. Also $UT_3(\Z[\sqrt{2}])$ is a 2-nilpotent group while the domain $P_\Gamma=\Z[\sqrt{2}]$ is not even 2-binomial (consider $\frac{\sqrt{2}(\sqrt{2}-1)}{2}$). \end{rem}
\subsection{Coordinatization of the action of the ring of scalars}\label{coord:sec}

Assume $A$ is a binomial domain. For a nilpotent $A$-group $G$ the action of the ring $A$
on various quotients of $G$ is coordinated, i.e. if $H_1\unlhd H_2$ are
$A$-subgroups of $G$, then the canonical epimorphism $H_2\rightarrow
H_2/H_1$ is an $A$-epimorphism. In this section we build subrings of
the ring $P_R$, which satisfy this property wherever the action of $P_R$ is defined. 

Denote by $L_R$ the set of canonical homomorphisms
$$e_i:R^l_{i}/R^l_{i+1}\rightarrow R^u_{i}/R^u_{i+1},$$induced by the
inclusions $R^l_i\rightarrow R^u_i$. There exists the largest subring
$PL_R$ of the ring $P_R$ (with respect to inclusion) such that all
homomorphisms of $L_R$ are $PL_R$-linear. 

Finally let us build a ring which acts simultaneously on quotients of some more central series,
for instance the upper and lower central series. Let
$$G=Q_1 > \ldots  > Q_d > Q_{d+1}=1, \qquad (Q)$$
be one more central series of the group $G$. Consider the bilinear maps, naturally defined using the commutation operation,  
$$g_i:\frac{G}{G'} \times \frac{Q^{u}_i}{Q^u_{i+1}} \to \frac{Q^l_{i+1}}{Q^l_{i+2}},$$ and let $S_Q$ be the set of all these maps. Define $Q^u$ and $Q^l$ similar to $R^u$ and $R^l$. Let $V_{R\cup Q}=V_R\cap V_Q$ and define 
$$F_{R\cup Q}: \frac{G}{V_{R\cup Q}}\times \left(R^u \oplus Q^u\right) \to R^l\oplus Q^l,$$
by $$F_{R\cup Q}(x, \sum_{i=1}^{c-1}x_i + \sum_{j=1}^{d-1}y_j) =\sum_{i=1}^{c-1}f_i(x_0,x_i) + \sum_{j=1}^{d-1}g_j(x_0,y_j),$$ where $x\in G/V_{R\cup Q}(G)$, $x_i\in R^u_i/R^u_{i+1}$, $y_j\in Q^u_j/Q^u_{j+1}$ and $x_0$ is an arbitrary pre-image of $x$ in $G/G'$. Needless to say that $F_{R\cup Q}$ is a full non-degenerate bilinear map. So the largest ring $P(F_{R\cup Q})$ with respect to which $F_{R\cup Q}$ is bilinear exists and similarly one can construct the largest subring $P_{R\cup Q}$ of  $P(F_{R\cup Q})$ which acts on all major direct factors of $R^u$, $Q^u$, $R^l$ and $Q^l$. Finally one can construct the subring $PL_{R\cup Q}$ similar to $PL_R$ above.

The two constructions introduced above can be easily generalized to an
arbitrary set $C$ of central series of the group $G$ and an
arbitrary set $L$ of canonical homomorphisms between the quotients of
series associated with $C$. In fact, let $S_R$ be the
bundle of mappings corresponding to the series $(R)$ from
$C$. Let $S_C=\bigcup_{R\in C}S_R$ and
$F_C$ be the corresponding bilinear mapping. Denote by
$PL_C\leq P_C$ the largest subring
(with respect to inclusion) such that all homomorphisms from
$L$ are $PL_{C}$-linear. The ring
$PL_C$ is a convenient tool because it allows
one to pass freely in reasoning from one series to another (in
$C$), specially when refining a series by means of another one is required.
\begin{prop}\label{noether-plc:prop} Let $G$ be finitely generated nilpotent $A$-group over a
Noetherian binomial domain $A$. Let $C$ be an arbitrary set of central
$A$-series of $G$ and $L$ an arbitrary subset of canonical homomorphisms
between quotients of series associated with $C$. Then
$PL_C$ is a finite dimensional commutative
associative $A$-algebra with a unit.\end{prop}
\begin{proof}By construction (see Section~\ref{bilin})
$PL_C$ is an $A$-subalgebra of the algebra of
$A$-endomorphisms $End_{A}(G/V_C)$ of the $A$-module
$G/V_{C}$. According to the construction
$V_C\geq G'$ and $G/G'$ is a
finitely generated $A$-group by hypothesis. Consequently the $A$-module
$G/V_C$ is finitely generated, so is the $A$-module
$End_A(G/V_C)$. Moreover this module is Noetherian since
$A$ is a Noetherian ring. Hence $PL_C$ is a
finitely generated $A$-module.

\end{proof}
\subsection{Maximally refined series}\label{cs}
We note that the ring $P_R$ acts exactly on all the quotients of the
upper series except on the one corresponding to the lowest gap $R^u_c=Z(G)> 1$, and on all
the quotients of the lower series except on the quotient of the upper gap $G> G'
=R^l_2$. On the other hand, $P_R$ acts on the
quotient $G/V_R$. Since $V_R \geq  G'\cdot  Z(G)$, then in the case of $Z(G)
\nleq G'$ there always remains a gap, $ G'\cdot Z(G)>G'$, on the
quotient of which the ring $P_R$ does not act. Similarly in a
2-nilpotent group $G$ if the gap $Z(G)\geq G'$ is nontrivial
($Z(G)\neq G'$), then the action of the ring $P_R$ on the quotient
$Z(G)/G'$ is undefined (see Example~\ref{2nil}).
In this section we introduce a simple construction which allows one to
maximally refine the series $(R^u)$ and $(R^l)$ so that the ring
$P_R$ acts on all the quotients of the refined series, except the
fixed special gap (if it is not trivial). For the refinement of
$(R^u)$ the special gap is $Z(G)\geq Z(G)\cap G'$, and for the
refinement of $(R^l)$ it is $G'\cdot Z(G) \geq G'$. Note that the
quotients of special gaps are isomorphic in the two cases. The special
gaps are essential and there is no way to get rid of them.

So let us start the refining process. We refine the series
$(R^u)$ in the last term $Z(G)>1$ with the help of the intermediate
terms $R^l_i\cap Z(G)$, $i=2,\ldots, c$, to get 
$$Z(G) \geq G'\cap Z(G)= R^l_2\cap Z(G)\geq \ldots \geq R^l_c\cap Z(G)\geq 1.$$
Thus we obtain \emph{the refined upper series (U(R)) associated with $(R)$}:
$$G=R^u_1\geq R^u_2 \geq \ldots \geq Z(G)\geq G'\cap Z(G) \geq \ldots \geq 1
\quad (U(R)).$$
Consider the set $E=\{ \epsilon_i: i=2, \ldots , c \}$ of the
canonical monomorphisms:
$$\epsilon_i: T_i= (R^l_i\cap Z(G))/(R^l_{i+1}\cap Z(G)) \rightarrow
R^l_i/R^l_{i+1}=S_i,$$
and denote by $AE_R$ the largest subring of $PL_R$, which leaves all
the images $\epsilon_i(T_i)$ in the $PL_R$-modules $S_i$ invariant. Then the action of $AE_R$
on $S_i$ induces an action of $AE_R$ on $T_i$ with the help of the
monomorphisms $\epsilon_i$ according to:
\begin{equation}\label{ActInvDefn}
\alpha x=_{\text{def}}\epsilon^{-1}(\alpha \epsilon(x)), \quad \alpha \in AE_R,  x\in
T_i.
\end{equation}

In this way, an action of the ring $AE_R$ is defined on all the
quotients of the series $(U(R))$, except on $Z(G)\geq
Z(G)\cap G'$, if it is non-trivial.

Dually starting from the lower series $(R^l)$ associated with $(R)$, we construct
\emph{the refined lower series $(L(R))$ associated with $(R)$}. Namely we refine the
series
$$G> V_R \geq G'=R^l_2\geq \ldots \geq R^l_c\geq R^l_{c+1}=1$$ in the gap $V_R\geq
G'$ by means of the upper associated series, i.e. we insert the series
$$V_R=(V_R\cap R^u_1)\cdot G'\geq (V_R\cap R^u_2)\cdot G'\geq\ldots \geq (V_R\cap R^u_c)\cdot G'=Z(G)\cdot G'.$$
As a result $(L(R))$ is of the form
\begin{align*}G\geq V_R \geq (V_R\cap R^u_2)\cdot G' &\geq \ldots \geq Z(G)\cdot G'\\
&\geq G' =R^l_2 \geq \ldots \geq R^l_c\geq R^l_{c+1}=1
\quad (L(R)).\end{align*}
For each $i$ consider the canonical epimorphism: 
$$\beta_i: X_i=(V_R\cap
R^u_i)/(V_R\cap R^u_{i+1}) \to ((V_R\cap R^u_i)\cdot G')/((V_R\cap R^u_{i+1})\cdot G')=Y_i,$$ and also consider the natural monomorphisms  $$\delta_i:(V_R\cap
R^u_i)/(V_R\cap R^u_{i+1})\to R^u_i/R^u_{i+1}=Z_i.$$ 

Denote by $AD_R$ the largest subring of $PL_R$, which leaves
all the images $\delta_i(X_i)$ in $Z_i$ invariant. Then one can use a definition similar to \eqref{ActInvDefn} to define an action of $AD_R$ on the $X_i$. Then the $Y_i$ can be turned into $AD_R$-modules via the $\beta_i$. Like $AE_R$, the
ring $AD_R$ acts on all the quotients of the series $(L(R))$,
except $Z(G)\cdot G'/ G'$, if it is non-trivial.

Let $A_R=AE_R \cap AD_R$. Then the ring $A_R$ acts, via the constructions above, on all the
quotients of the series $(U(R))$ and $(L(R))$ except the ones corresponding to the special
gaps. Moreover in both cases the special gaps give the same quotients,
$$(Z(G)\cdot G')/G'\cong Z(G)/(Z(G)\cap G')$$ up to isomorphism.

We close the section with the following proposition. Its proof uses Lemma~\ref{Rulbinom:lem} and is similar to that of Proposition~\ref{noether-plc:prop} and is omitted here.

\begin{prop}\label{ringf}Let $A$ be binomial domain and let $G$ be a nilpotent $A$-group. If $(R)$ is a central series of $A$-subgroups of
$G$, then all the terms of $(L(R))$ and $(U(R))$ are $A$-subgroups and $A_R$ is an associative commutative $A$-algebra with a unit. If in addition $A$ is Noetherian and $G$ is finitely generated as an $A$-group, then $A_R$ is finite dimensional as an $A$-algebra. \end{prop}

\section{ Some logical invariants of finitely generated nilpotent groups}\label{logical:sec}

In this section we study the logical properties of the constructions introduced in the previous section. In particular we will show that all the constructions associated with the bilinearization of $G$, are definable by formulas of the language of groups. Most of the material here is folklore in the context of finitely generated nilpotent groups. We consider the slightly more general case of finitely generated nilpotent $A$-groups where $A$ is a binomial domain. 

\subsection{Uniform interpretability of the bilinearization}

In this subsection it will be shown that bilinearization of an arbitrary finitely generated nilpotent group $G$ is  absolutely interpretable in the group $G$ by a system of formulas of the signature of groups. Besides, the same formulas absolutely interpret bilinearization in any group $H$ which is elementary equivalent to $G$. In other words, bilinearization in the group $G$ is a logical invariant of $G$. Moreover, this is true for arbitrary finitely generated $A$-groups over a binomial domain $A$, e.g.  for finitely generated nilpotent pro-$p$-groups, unipotent $k$-groups over a field $k$ of zero characteristic. Most of the arguments here are quite well-known and standard. 

 Below we will prove that $G'$ as well as other verbal subgroups of a finitely generated $A$-group $G$ are definable in $G$. Definability of verbal subgroups of a finitely generated $A$-exponential nilpotent group $G$ is related to the notion of finite width. Assume a subgroup $N$ of the group $G$ is generated by the set $X$. We say that $N$ is of \textit{width} $s$ (with respect to $X$), if any element of $N$ can be represented as a product of no more than $s$ elements of the set $X$ and their inverses, $s$ being the minimal number with this property. The width of the subgroup $N=\langle X\rangle$ will be denoted by $s_X(N)$ or simply $s(N)$.

Let $v(x_1,\ldots , x_n) = v(\bar{x})$ be a group word. The set
$$V= \{v(\bar{g}): \bar{g}\in G^n\}$$ is called a \textit{value set} of $v$ in $G$. The subgroup $v(G)$ generated by the set $V$ is called a \textit{verbal subgroup}, defined by the word $v$. By the width of the subgroup $v(G)$ we mean the width with respect to the value set $V$.

Let $\varphi(x)$ be a formula of the language of groups. Consider the functor $f_{\varphi}$ which associates to any group $G$ its subgroup $f_{\varphi} (G)$, generated by the definable subset
$$\varphi(G)=\{g\in G :  G\models \varphi(g)\}.$$

\begin{prop} \label{prop61}Let $G$ be a group. If the subgroup $f_{\varphi}(G)$ is of finite width, then $f_{\varphi}(G)$ is definable in $G$ uniformly with respect to $Th(G)$.\end{prop}
\begin{proof} Let $s$ be the width of the subgroup $f_{\varphi}(G)$ with respect to $\varphi(G)$. Then the subgroup $\varphi(G)$ is defined in $G$ by the formula:
$$\Phi_s(x) = \exists x_1,\ldots , \exists x_s (x=\prod_{i=1}^s x_i\wedge \bigwedge^s_{i=1}(\varphi(x_i) \vee \varphi(x_i^{-1}))).$$
The group $G$ and hence any group $H\equiv G$ satisfies the sentence
$$\forall x(\Phi_{s+1}(x)\rightarrow \Phi_{s}(x)),$$
which allows to contract any product of elements of the set $\varphi(G) \cup \varphi(G)^{-1}$ to a product of no more than $s$ factors of $\varphi(G) \cup \varphi(G)^{-1}$. Consequently, if $H\equiv G$, then the width of $f_{\varphi}(H)$ is not greater than $s$, and hence it is exactly $s$. Therefore $\Phi_s(x)$ defines $f_{\varphi}(H)$ in all models $H$ of $Th(G)$.

\end{proof}

\begin{cor}Any verbal subgroup $v(G)$ of finite width is definable in $G$ uniformly with respect to $Th(G)$.\end{cor}
\begin{lem} \label{lem61}Let $G$ be a $c$-nilpotent $A$-exponential group generated by the set $X=\{x_1,\ldots ,x_n\}$. Let $N$ be a definable normal subgroup of $G$. Then the subgroup $[N, G]$ is of width not greater than $n(c-1)$ with respect to the set of generators
$$[N,X]=\{[y,x]:y\in N, x\in X\}.$$\end{lem}
\begin{proof} Let $R_0=N$, $R_i=[N, \underbrace{  G,\ldots ,G  }_{\textrm{
$i$-times}}]$ , $i=1,\ldots, c$, and consider the central series of subgroups
$$N=R_0\geq [N,G]=R_1\geq R_2 \geq \ldots \geq R_c=1.$$
Since $[R_i,G]=R_{i+1}$, then for $y\in R_i$, $g_1,g_2\in G$ and $\alpha \in A$ we have
\begin{equation}\label{Hall}\begin{split}
[y,g_1g_2]R_{i+2}&= [y,g_1][y,g_2]R_{i+2} \\
[y,g_1^\alpha]R_{i+2}&=[y^\alpha,g_1]R_{i+2}.
\end{split}\end{equation}
Hence for $y_k \in N$, $g_k\in G$, decomposing $g_k$ with the help of generators from $X$ and using congruences \eqref{Hall}, we obtain:
$$\prod_{k=1}^n[y_k,g_k]=\prod_{k=1}^n[y_k(0),x_k]r_2,$$
where $y_k(0)\in R_0=N$, $r_2\in R_2$. Continuing this process for $r_2$, we obtain by induction
$$\prod_{k=1}^n[y_k,g_k]=\prod_j^{c-1}\prod_{i=1}^n[y_i(j),x_i],$$
where $y_i(j)\in R_j \leq N$. Consequently, the width of $[N,G]$ with respect to the set of generators $[N,G]$ does not exceed $(c-1)n$ . 

\end{proof}
\begin{cor}\label{corlem1} Assume $G$ is nilpotent group which is finite generated as an $A$-group. Then each term of the lower central series of $G$ is a subgroup of finite width. \end{cor}
\begin{prop} \label{prop62}Assume $G$ is a finitely generated $A$-exponential group. Then:
\begin{enumerate}
\item each term of the upper central series
$$G=Z_c(G)> Z_{c-1}(G) > \ldots > Z_1(G) > Z_0(G)=1 \qquad                 (Z)$$
is definable in $G$ uniformly with respect to $Th(G)$,
\item each term of the lower central series
$$G=\Gamma_1(G)> \Gamma_2(G) > \ldots > \Gamma_c(G) > \Gamma_{c+1}=1$$
is definable in $G$ uniformly with respect to $Th(G)$.
\end{enumerate}\end{prop}
\begin{proof} Define the formulas $\Phi_i(x)$, $i=0, 1, \ldots, c$,  recursively as follows.
$$\Phi_0(x)= (x=1), \quad \Phi_i(x)=\forall y~\Phi_{i-1}([x,y]),\text{ if, } 1\leq i \leq c.$$
Clearly $\Phi_i(x)=\forall y_i,\ldots,y_1 ([x,y_i,\ldots, y_1]=1)$ and so, for each group $G$, each $x\in G$ satisfies $\Phi_i$ in $G$ if and only if $x\in Z_i(G)$. So the $Z_i(G)$ are definable in $G$ uniformly with respect to $Th(G)$. Part (2) follows from Proposition~\ref{prop61} and Corollary~\ref{corlem1}. 

\end{proof}
 
\begin{lem}\label{inter-quo:lem} Assume $N_i\unlhd G_i$, $i=1,2, \ldots, n$, for some natural number $n$, are definable subgroups of $G$ uniformly with respect to $Th(G)$. Then,
\begin{enumerate}
\item all quotient groups $G_i/N_i$ are interpretable in $G$ uniformly with respect to $Th(G)$,
\item the direct sum $\bigoplus_{i=1}^n G_i/N_i$ is interpretable in $G$ uniformly with respect to $Th(G)$.  
\end{enumerate}
\end{lem}
The proof of the lemma is an elementary exercise in model theory so it is omitted.
\begin{prop}\label{prop63}Assume $G$ is a finitely generated $A$-exponential nilpotent group, and each term of the $A$-series
$$G=R_c>R_{c-1}>\ldots R_1>R_0=1 \qquad (R)$$
is definable in $G$ uniformly with respect to $Th(G)$ . Then:
\begin{enumerate}
\item all the terms $R^u_i$ of the upper series associated with $(R)$
are definable in $G$ uniformly with respect to $Th(G)$,
\item all the terms $R^l_i$ of the lower series associated with $(R)$ are definable in $G$ uniformly with respect to $Th(G)$,
\item the system $\mathfrak{U}(F_R)$ corresponding to the bilinear mapping $F_R$ is interpretable in $G$ uniformly with respect to $Th(G)$,
\item the rings $P_R$ and $PL_R$ are non-empty and their actions on the quotients $R^u_i/R^u_{i-1}$ and $R^l_{i}/R^l_{i-1}$
are interpretable in $G$ uniformly with respect to $Th(G)$.\end{enumerate}\end{prop}
\begin{proof} Parts (1) and (2) immediately follow from Lemma~\ref{lem61} and Proposition~\ref{prop61}. To prove part (3) note that by parts (1) and (2) and Lemma~\ref{inter-quo:lem} $G/V(R)$, $R^u$ and $R^l$ are all absolutely interpretable in $G$. Moreover we use only commutators $[x,y]$ to define the bilinear map $F_R$. To verify part (4) note that $F_R$ has a finite complete system and is of finite width since all the modules $G/V(R)$, $R^u$ and $R^l$ are finitely generated $A$-modules and $F_R$ is $A$-bilinear. It is by construction full and non-degenerate. So $F_R$ satisfies the hypotheses of Theorem~\ref{ringinter}. Therefore the largest ring of scalars $P(F_R)$ exists and more importantly the structure $\mathfrak{U}_{P(F_R)}(F_R)$ is absolutely interpretable in $\mathfrak{U}(F_R)$ which by part (3) is absolutely interpretable in $G$. Interpretations are uniform with respect to $Th(G)$. This in turn implies that the largest subring $P_R$ of $P(F_R)$ stabilizing these quotients is interpretable in $G$.  Finally the homomorphisms
$$e_i:R^l_i/R^l_{i-1}\rightarrow R^u_i/R^u_{i-1},$$ are interpretable in $G$, which implies that $PL_R$ is interpretable in $G$. By definition all these rings include the identity map and are all nontrivial. 
\end{proof}

\noindent\emph{Proof of Proposition~\ref{prop64}.} The proof is similar to that of Proposition~\ref{prop63}.

\qed

\section{Interpretations in finite dimensional commutative algebras
and elementary equivalence}\label{Z-interpret:sec}
In this section we describe by first order formulas some algebraic
invariants of any associative commutative ring $A$ with finitely
generated additive group $A^+$. In particular we provide proofs of  Theorem~\ref{elemmod:thm} and \ref{elem-iso-alg:thm}. 
\subsection{Interpretability of decomposition of zero into the product
of prime ideals with fixed characteristic}\label{Prime-decom:sec}

Let $A$ be an arbitrary associative commutative ring with a
unit. Suppose that we have a decomposition of zero into the product of
finitely generated prime ideals:
$$0=\p_1\cdot \p_2 \cdots \p_m,\qquad (\mathfrak{P})$$
Let $Char(\p_i)=\lambda_i$ be the characteristic of the integral
domain $A/\p_i$ and $$Char(\mathfrak{P})=(\lambda_1,\ldots ,
\lambda_m).$$ The purpose of this subsection is to obtain a formula
interpreting the decomposition of type $(\mathfrak{P})$ in $A$ with the
fixed characteristic $Char(\mathfrak{P})$, where the interpretation is uniform with respect to $Th(A)$.

A sequence of lemmas will follow. We omit some proofs as they are obvious.

\begin{lem} \label{P1} Consider the formula
$$Id(x,\bar{y})=\exists z_1,\ldots, \exists z_n (x=y_1z_1+\ldots + y_nz_n).$$
For any tuple $\bar{a}=(a_1,\ldots, a_n)\in A^n$ the formula $Id(x,\bar{a})$
defines in $A$ the ideal $id(\bar{a})$, generated by the elements
$a_1,\ldots, a_n$.\end{lem}

\begin{lem}\label{P2} The formula
$$P(\bar{y})=\forall x_1,\forall x_2(Id(x_1x_2,\bar{y})\rightarrow
(Id(x_1,\bar{y})\vee Id(x_2,\bar{y})))$$
is true for the tuple $\bar{a}$ of elements of the ring $A$ if and
only if the ideal $id(\bar{a})$ is prime.\end{lem}
\begin{lem}\label{P7} There exists a formula $Id_i(x,\bar{y}_1,\ldots , \bar{y}_i)$,
such that for any tuples $\bar{a}_1$, \ldots , $\bar{a}_i$, $Id_i(x,\bar{a}_1,\ldots , \bar{a}_i)$ defines the ideal
$\p_1\cdots \p_i$ in $A$ where $\p_k=id(\bar{a}_k)$. \end{lem}
Indeed the ideal $\p_1\cdots \p
_i$ is generated by all the products of the form $y_1\cdots y_i$
where $y_k$ is an element of the tuple $\bar{a}_k$ and the number of such products
is finite.
\begin{lem}\label{P3} The formula:
$$D(\bar{y_1},\ldots,
\bar{y_m})=\forall x( \bigwedge_{i=1}^mP(\bar{y}_i) \wedge Id_m(x,\bar{y}_1,\ldots , \bar{y}_m)\rightarrow
x=0)$$
is true for tuples $\bar{a}_1,\ldots, \bar{a}_m$ if and only if the
ideals $\p_i=Id(\bar{a}_i)$ satisfy the decomposition
$(\mathfrak{P})$.\end{lem}

\begin{lem}\label{P4} The formula
$$D_{\Lambda}(\bar{y}_1,\ldots , \bar{y}_m)=D(\bar{y}_1, \ldots
\bar{y}_m) \wedge \bigwedge_{i=1}^m \forall x Id(\lambda_ix,\bar{y}_i)\wedge\bigwedge_{i=1}^m
\exists z \neg Id(z,\bar{y}_i)$$
where $\Lambda=(\lambda_1,\ldots \lambda_m)=Char(\mathfrak{P})$ is true
for tuples $\bar{a}_1,\ldots ,\bar{a}_m$ of elements of $A$ if and only
if all the following statements hold:
\begin{itemize}
\item the ideals $\p_i=id(\bar{a}_i)$ satisfy the decomposition
$(\mathfrak{P})$,
\item if $\lambda_i>0$ then $Char(A/\p_i)=\lambda_i$, 
\item the integral domains $A/\p_i$ are all non-zero.
\end{itemize}\end{lem}

Denote by $0(\mathfrak{P})$ the number of zeros in the tuple
$(\lambda_1,\ldots , \lambda_m)$.
\begin{lem}\label{decom}Let $\mathfrak{P}=(\p_1,\ldots, \p_m)$ be a collection of finitely generated prime ideals of the ring $A$, satisfying the
decomposition $(\mathfrak{P})$ and possessing the least number
$0(\mathfrak{P})$ among all such decompositions. Then for any ring $B$
if $A\equiv B$, then the formula $D_{\Lambda}(\bar{y}_1,\ldots ,
\bar{y}_m)$ is true in $B$ on tuples $\bar{b}_1$, \ldots, $\bar{b}_m$,
if and only if:
\begin{enumerate}

\item $Id(x,\bar{b}_i)$ defines the prime ideal $\q_i=id(\bar{b}_i)$,
\item $0=\q_1\cdot \q_2\cdots \q_m$,
\item $Char(B/\q_i)=\lambda_i$, \quad $i=1, \ldots ,m$.
\end{enumerate}\end{lem}
\begin{proof} Items 1 and 2 follow from Lemma~\ref{P3}. If $\lambda_i>0$ then
Char$(B/q_i)=\lambda_i$ according to Lemma~\ref{P4}. Consequently,
$0(\mathfrak{P})\geq 0(\mathfrak{Q})$ where $\mathfrak{Q}=(\q_1,\ldots
, \q_m)$. If $0(\mathfrak{P})>0(\mathfrak{Q})$ then starting from
$\mathfrak{Q}$ we construct the formula $D_{\mu}$,
$\mu=char(\mathfrak{Q})$. From $A\equiv B$ and Lemma~\ref{P4} we obtain that
there exists a tuple $\mathfrak{P}'=(\p'_1,\ldots, \p'_m)$ such that
$0(\mathfrak{P}')\leq 0(\mathfrak{Q})<0(\mathfrak{P})$ which contradicts the choice of $\mathfrak{P}$. Consequently
$0(\mathfrak{P})=0(\mathfrak{Q})$ and hence
$Char(\mathfrak{P})=Char(\mathfrak{Q})$. The proposition is
proved.

\end{proof}
\begin{rem}\label{decom:rem} Any Noetherian commutative associative ring with a unit
possesses a decomposition of zero $0=\p_1\ldots \p_m$, satisfying the
assumptions of Proposition~\ref{decom}.\end{rem}

\begin{prop} \label{primary:thm} For any Noetherian associative commutative ring $A$ with a
unit, there exists an interpretable decomposition of zero into a product of prime ideals, where the interpretation is
uniform with respect to $Th(A)$.\end{prop} 
\begin{proof} The proposition is a direct corollary of Proposition~\ref{decom} and Remark \ref{decom:rem}.\end{proof}

\subsection{The case of a ring with a finitely generated additive group}\label{elem-eq-rings:sec}
Let $A$ be a commutative associative ring with unit and a finitely
generated additive group $A^+$. We shall denote by $r(A)$ the minimal number of generators of $A^+$ as an abelian group, say, the number of cyclic factors in the invariant decomposition of $A^+$. In case that $M$ is a finitely generated $A$-module where $A$ is as above the minimal number of generators of $M$ as an abelian group is denoted by $r(M)$, while the minimal number of generators for $M$ as an $A$-module is denoted by $r_A(M)$.

\begin{lem}\label{P9} There exists a sentence $ch_{\lambda}$ of the language of rings such
that for any integral domain $A$ with finitely generated additive group
$A^+$:
$$char(A)=\lambda \Leftrightarrow A\models ch_{\lambda}.$$\end{lem}
\begin{proof} To prove the claim notice that if $\lambda$ is a prime then we can set
$ch_{\lambda}= \forall x (\lambda x=0)$. For $\lambda =0$ it is
enough to note that for the integral domain $A$, $char(A)=0$ if and
only if $2\neq 0$ and $1/2\notin A$. In fact if $char(A)=0$ then $2\neq
0$ and if $1/2\in A$ then $A\geq \mathbb{Z}[1/2]$ but
$\mathbb{Z}[1/2]$ is not finitely generated. Contradicting with the
assumption that $A^+$ is finitely generated. Conversely if $char(A)=p\neq 2$, then
$pA=0$. So $A$ contains the finite field $\Z/p\Z$ and so
$1/2\in A$.

\end{proof}
\begin{lem}\label{P10} Let $A$ be a commutative associative ring with unit and $r(A)=n$. Then there exists a sentence $\varphi_{n}$ of the language of rings such that $ A
\models \varphi_{n}$ and for any commutative associative ring $B$ with unit,
$$B \models \varphi_n \Leftrightarrow r(B)\leq n.$$\end{lem}

\begin{proof} 
Let us first assume that $A$ is an integral domain. By  Lemma~\ref{P9} there is a sentence $ch_{\lambda}$ that defines the characteristic $char(A)=\lambda$ of $A$ in the language of rings. If   $char(A)=\lambda\neq
0$ then $A$ is finite and $\varphi_n$ will say that $ch_\lambda$ and $A$ does not
have more than $\lambda^n$ elements. If $char(A)=0$ then
$r(A)=n$ if and only if $|A^+/2A^+|=2^n$. So in this case $\varphi_n$ will say that $ch_0$ and there are precisely $2^n$ distinct elements in $A$ modulo $2A$. 

Now assume $A$ is not necessarily an integral domain. Then $A$ admits a decomposition of zero $$0=\p_1\cdot \p_2 \cdots \p_m,\qquad (\mathfrak{P})$$ with  $\Lambda=char(\mathfrak{P})$. Set $O_i=\p_0\cdots\p_i$, where $\p_0=A$. Set also $\bar{O}_i=O_i/O_{i+1}$. Note that $r(A)$ is bounded by 
$$\sum_{i=0}^{m-1}r(\bar{O}_i).$$
So it is enough to come up with sentences $\varphi_i$ each expressing a bound for $r(\bar{O}_i)$. By lemma~\ref{P4} there are tuples of elements $\bar{a}_i$ ,$i=1, \ldots, m$ satisfying $D_{\Lambda}(\bar{a}_1, \ldots \bar{a}_m)$. Moreover if $B$ is any ring similar to $A$ with tuples of elements $\bar{b}_1, \ldots, \bar{b}_m$ which satisfy $D_\Lambda(b_1,\ldots ,b_m)$ then by Lemma~\ref{P5}, $B$ has a decomposition of zero $\mathfrak{Q}$ with same exact properties of $\mathfrak{P}$. Moreover the formula $id_i(x,\bar{a}_1, \ldots \bar{a}_i)$ from Lemma~\ref{P1} defines $O_i$ in $A$ and $id_i(x,\bar{b}_1, \ldots \bar{b}_i)$ defines similar term in $B$. The quotients  $\bar{O}_i$ are finitely generated $A$-modules over the integral domains $A/\p_i$. Assume $r(A/\p_{i+1})=n_i$ and $r_{A/\p_{i+1}}(\bar{O})=s_i$. Note that $r(\bar{O}_i)\leq n_is_i$. So it is enough to define $n_i$ and $s_i$ in the language of rings. By definability of the $\p_i$ and the $O_i$ it is easy to write a sentence in the language of rings saying that $r_{A/\p_{i+1}}(\bar{O}_i)\leq s_i$. By the first paragraph of this proof and definability of $\p_{i+1}$ there is also a sentence in the language of rings saying that $r(A/\p_{i+1})\leq n_i$. Note that the same formulas work for a ring $B$ as above. 
\end{proof}
\begin{cor} \label{10b} Assume $\mathcal{K}$ is the class of all associative commutative unitary rings $A$ with unit and $r(A)<\infty$. Assume $\mathcal{I}_n$ is the subclass of $\mathcal{K}$ consisting of all integral domains $A$ of characteristic zero with $r(A)\leq n$, for some natural number $n>0$. Then there exists a sentence $\Phi_n$ of the language of rings such that for any $A\in \mathcal{K}$
$$A\models \Phi_n \Leftrightarrow A\in \mathcal{I}_n.$$\end{cor}
\begin{proof} A ring $A$ being an integral domain is axiomatizable by one ring theory sentence. The formula $ch_0$ from Lemma~\ref{P9} is true in any $A\in \mathcal{I}_n$ and conversely implies that $A\in \mathcal{K}$ has characteristic 0 once $A$ satisfies it. The formula $\varphi_n$ from Lemma~\ref{P10} is satisfied by any $A\in \mathcal{I}_n$ and conversely will force $r(A)\leq n$ for any $A\in \mathcal{K}$ satisfying it. The conjunction of these sentences is the desired one.\end{proof}
\begin{cor} \label{10a} Let $A\in \mathcal{I}_n$. Then, there exists a formula $\phi_{\Z}$ of the language of rings such that 
$$A \models \phi_\Z \Leftrightarrow A\cong \Z.$$ 
\end{cor}
\begin{proof} By Corollary~\ref{10b} the formula $\Phi_1$ characterizes members of $\mathcal{I}_1$ among those of $\mathcal{K}$. But $\mathcal{I}_1$ has only one member up to isomorphism, namely $\Z$. So we may set $\phi_\Z=\Phi_1$. \end{proof}

\begin{lem}\label{julia:lem} Consider the class $\mathcal{I}_n$ introduced in Corollary~\ref{10b}. Then there exists a formula $R_n(x)$ defining
the subring $\mathbb{Z}\cdot 1_A$ in any member $A$ of $\mathcal{I}_n$. \end{lem}
\begin{proof} We need to note that the field of fractions $F$ of $A$ is an extension of field of rationals $\Q$ with dimension $n$ over $\Q$. So $F$ is a field of algebraic numbers of finite degree over $\Q$. Now by Theorem on page 956 of \cite{julia} the ring of integers $\Z$ is definable in $F$ by a formula $\Phi_{\Z(F)}(x)$. An inspection of the proof shows that the formula defines $\Z$ in any algebraic extension $K$ of $\Q$ with $[K:\Q]\leq [F:\Q]=r(A)=n$. Moreover $F$ is uniformly interpretable in $A$. Though elementary, let us elaborate on this claim here a bit. Recall that $F$ is realized as $X/\sim$ where
$$X=\{(x,y):x\in A, y\in A\setminus \{0\}\},$$
and $\sim$ is the equivalence relation on $X$ defined by 
$$(x,y)\sim (z,w) \Leftrightarrow xw=yz.$$
Addition and multiplication are defined on $X/\sim$ in the obvious manner using addition and multiplication on $A$. The same formulas interpret the field of fractions $K$ of any integral domain of characteristic zero $B$ in $B$. So combining the results here we have an interpretation of $\Z$ in $A$. 

But the above interpretation of $\Z$ in $A$ also provides a formula defining $\Z$ (as a subset of $A$) in $A$ in the following way. Note that there is an interpretable monomorphism $\mu: A \to F$ defined by $\mu (a)=[(a,1)]$ where $|F|=X/\sim$ is considered as the set of equivalence classes $[(x,y)]$ described above. Now the copy of $\Z$
sitting in $F$ is included in the image of $\mu$ so the copy of $\Z$ in $A$ is a definable subset of $A$ as $\mu^{-1}(\Phi_{\Z(F)}(\mu(A))$. Since by Corollary~\ref{10a} $\Z$ is axiomatizable in $\mathcal{I}_n$ by one formula,  there exists a formula defining $\Z$ in any member of $\mathcal{I}_n$.   
\end{proof} 

\begin{lem}\label{P5} There exists a formula $R_{n,\Lambda}(x,\bar{y})$ such that for any commutative associative 
ring $A$ with unit and $r(A)\leq n$ and for any prime ideal
$\p=id(\bar{a})$ of $A$ if $char(A/\p)=\lambda$ then the formula
$R_{n,\Lambda}(x,\bar{a})$ defines the subring
$$\mathbb{Z}\cdot 1+\p=\{z\cdot 1+x:z\in \mathbb{Z}, x\in \p\},$$
in $A$. \end{lem}

\begin{proof} Indeed the ideal $\p=id(\bar{a})$ is defined in $A$ by the formula
$Id(x,\bar{a})$. Consequently the ring $A/\p$ and the canonical
epimorphism $A \rightarrow A/\p$ are interpretable in $A$. Therefore
to obtain $R_{n,\Lambda}(x,\bar{y})$ it is sufficient to define the
subring $\mathbb{Z}\cdot 1$ in $A/\p$. In the case of $Char (A/\p)=0$ we use the formula $R_n(x)$ from Lemma~\ref{julia:lem}. As for the case of $char(A/\p)>0$ the set $\mathbb{Z}\cdot 1 +\p$ is finite
in $A/\p$ and hence definable in $A/\p$. 

\end{proof}

\begin{lem}\label{P6} Assume $A$ is a commutative associative ring with unit and $r(A)\leq n$, admitting a decomposition
$$0=\p_1\ldots \p_m, \quad (\mathfrak{P}),$$
into prime ideals $\p_i=id(\bar{a}_i)$ with $char(A/\p_i)=\lambda_i$, $i=1, \ldots ,
n$. Let $\Lambda=(\lambda_1, \ldots , \lambda_m)$. Then there exists a first-order formula of the language of rings $R_{n,\Lambda}(x, \bar{y}_1, \ldots ,
\bar{y}_m)$ where  such that
 then the formula $R_{n,\Lambda}(x,\bar{a}_1, \ldots, \bar{a}_m)$
defines in $A$ the subring
$$A_{\mathfrak{P}}=\bigcap_{i=1}^m(\mathbb{Z}\cdot 1 + \p_i),$$
\end{lem}
\begin{proof} For each $1\leq i \leq m$ consider the formula $R_{n,\lambda_i}(x,\bar{y}_i) $ introduced in Lemma~\ref{P5}. So we can set
$$R_{n,\Lambda}(x,\bar{a}_1, \ldots, \bar{a}_m)=\bigwedge_{i=1}^mR_{n,\lambda_i}(x,\bar{y}_i).$$
\end{proof} 
To a decomposition of $0$ in $A$ as above we associate the
series of ideals $$A> \p_1>\p_1\p_2>\ldots> \p_1\cdots \p_m =0$$ of the ring $A$
which will be called a $\mathfrak{P}$-series. The ring
$A_{\mathfrak{P}}$ from Lemma~\ref{P6} acts on all the quotients $\p_1\cdots \p_i/\p_1\cdots
\p_{i+1}$
 as each subring $\mathbb{Z}\cdot 1 +\p_{i+1}$ acts on the quotient  $\p_1\cdots \p_i/\p_1\cdots
\p_{i+1}$ for each $i=1, \ldots, m$.

Recall that $L_2$ is the language of two-sorted modules, say, $\langle M, A,
s \rangle$. We drop $s$ from our notation and denote the structure only by $\langle M, A\rangle$.  
If $A$ is a commutative associative ring with unit admitting a decomposition of zero $\mathfrak{P}$, then $\mathfrak{P}$-series of the
ring $A$ induces a series of $A$-modules
$$M\geq \p_1M\geq \p_1\p_2M\geq \ldots \geq \p_1\cdots \p_mM=0,$$
which will also be called a \emph{$\mathfrak{P}$-series for the $A$-module $M$} or a
special series for $M$.
The following lemma is a direct corollary of Proposition~\ref{decom}.
\begin{lem}\label{P8} There exists a formula $\phi_i(x,\bar{y}_1, \ldots ,\bar{y}_i)$
such that if $\p_1\cdots \p_m=0$ is a decomposition of zero in the ring
$A$ and $\p_k=id(\bar{a}_k)$, then $\phi_i(x,\bar{a}_1, \ldots ,
\bar{a}_i)$ defines the submodule $M_i=\p_1\cdots \p_iM$ in the two-sorted model $\langle M,A \rangle$.\end{lem}

The following proposition collects the main results of this section so far. 
\begin{prop}\label{big} 
Assume $A$ ($B$) is commutative associative ring with unit  
 and finitely generated additive group $A^+$ (resp. $B^+$). Let $M$ (resp. N) be a finitely generated $A$-module (resp. $B$-module). Assume there are tuples $\bar{a}_1$,$ \ldots $, $\bar{a}_m$ of elements of $A$ which satisfy $D_{\Lambda}(\bar{x}_1, \ldots , \bar{x}_m)$. Moreover assume $\langle M, A\rangle \models \varphi_n$. Then the following hold uniformly with respect to all models $\langle N, B\rangle$ of $Th(\langle M, A\rangle)$  containing tuples $\bar{b}_{i}$ which satisfy $D_{\Lambda}(\bar{b}_1, \ldots , \bar{b}_m)$.
\begin{enumerate}

\item $Id(x,\bar{b}_i)$ defines in $B$ the prime ideal
$\q_i=id(\bar{b}_i)$;
\item If $\mathfrak{Q}=(\q_1, \ldots , \q_m)$ then
$char(\mathfrak{q})=\Lambda$;
\item $0=\q_1\cdots \q_m$;
\item $\phi_i(x,\bar{q}_1, \ldots , \bar{q}_i)$ defines the $i$-th
term $N_i=\q_1\cdots \q_iN$ of the special $\mathfrak{Q}$-series for $N$ in $\langle N,B \rangle$;
\item $r(B)\leq n$.
\item The formula $R_{n,\Lambda}(x,\bar{b}_1, \ldots , \bar{b}_m)$, defines the subring
$$B_{\mathfrak{Q}}=\bigcap_{i=1}^m(\mathbb{Z}\cdot 1+\q_i),$$ in $B$.

\end{enumerate}\end{prop}
\begin{proof} Items (1)-(4) follow from  Lemma~\ref{decom},
\ref{P6} and  \ref{P9}. Part (5) follows directly from Lemma~\ref{P9}. To prove (6.) we note that by (5), $r(B)\leq n$. So the statement follows from Lemma~\ref{P8}.
 \end{proof}
Finally we are ready to prove the main technical result of this section. By a \emph{$\Z$-pseudo-basis} for finitely generated abelian group $M$ we simply mean a minimal generating set for $M$ as an abelian group. Assume $\bar{u}=(u_1, \ldots , u_s)$ is an ordered $\Z$-pseudo-basis for $M$ and let $M_i$ be the subgroup of $M$ generated by $u_i, \ldots, u_s$. The \emph {the period $e_i$} of $u_i$ is the order of the cyclic group $M_i/M_{i+1}$ if $M_i/M_{i+1}$ is finite, and we set $e_i=\infty$ if the corresponding quotient is infinite. 
\begin{prop} \label{basedef} Let $M$, $N$, $A$ and $B$ satisfy conditions of Proposition~\ref{big}.  Let the collection of prime ideals $\mathfrak{P}=(\p_1,\ldots ,
\p_m)$, $\p_i=id(\bar{a}_i)$, satisfy the usual conditions, say as in
Proposition~\ref{big}. Assume $\bar{c}=(c_1, \ldots , c_n)$ is a $\Z$-pseudo-basis of period $\bar{f}=(f_1, \ldots f_n)$ associated to the   $\mathfrak{P}$-series of $A$, and $\bar{u}=(u_1, \ldots, u_s)$ is a pseudo-basis of period $\bar{e}=(e_1, \ldots , e_s)$ associated with the $\mathfrak{P}$-series for $M$. Then there exists a
formula $$\phi_{\mathfrak{P},n}(x_1, \ldots , x_n,y_1, \ldots , y_s, \bar{y}_1, \ldots ,
\bar{y}_m)$$ defining in the two-sorted structure $M_A^*=\langle M, A, \bar{a}_1,\ldots ,
\bar{a}_m\rangle$ the set of all $\Z$-pseudo-bases of periods $\bar{f}$ and $\bar{e}$ associated
with the $\mathfrak{P}$-series for $A$ and $M$, respectively. Moreover the formula $\phi_{\mathfrak{P},n}(\bar{x}, \bar{y}, \bar{b}_1, \ldots, \bar{b}_m)$ defines in the model $N_B^*=\langle N,B, \bar{b}_1, \ldots \bar{b}_m\rangle$ the set of all $\Z$-pseudo-bases $\bar{d}$ and $\bar{v}$ of $N$ associated to the corresponding special $\mathfrak{Q}$-series of $B$ and $N$ uniformly for all models $N_B^*$ of $Th(M_A^*)$,  if $\bar{b}_1$, \ldots , $\bar{b}_m$ satisfy the formula $D_{\Lambda}$.  
\end{prop}

\begin{proof} By Proposition~\ref{big} the models
$\langle M_i, A_{\mathfrak{P}}, \bar{a}_i\rangle$ are definable in
$M_A^*$ uniformly with respect to all models $N_B^*$, with $N$ and $B$ satisfying the  hypotheses, of $Th(M_A^*)$. So it suffices to find a formula $\phi_i$ for each $i$,
defining a basis for $M_i/M_{i+1}$ of fixed period
$\bar{e}_i$.  the model
$\langle M_i, A_{\mathfrak{P}}\rangle$ is interpretable in $\langle
M,A\rangle$ with the help of any tuples of generating elements
$\bar{a}_1$, \ldots , $\bar{a}_m$ satisfying the formula
$D_{\Lambda}$. Since $M_{i+1}=\p_{i+1}M_i$, the model $\langle \ov{M}_i,A_i
\rangle$ where $\ov{M}_i=M_i/M_{i+1}$ and
$A_i=A_{\mathfrak{P}}/(\p_{i+1}\cap A_{\mathfrak{P}})$ is
obviously interpretable in $\langle M_i, A_{\mathfrak{P}}\rangle$ with the
help of $\bar{a}_{i+1}$. In the view of the fact that $A_i$ is
either $\mathbb{Z}$ or the finite field $\mathbb{Z}/p\mathbb{Z}$ and
the action of $A_i$ on
$\ov{M}_i$ is interpretable in $\langle M,A\rangle$ it is easy to write out a formula
defining all bases of $\ov{M}_i$ of given period $\bar{e_i}$ and
thus to construct the desired formula $\phi_i$. Again $\phi_i$ depends on the tuples $a_i$ as far as they satisfy $D_\Lambda$. So again by Proposition~\ref{big} the formulas $\phi_i$ define all $\Z$-pseudo-bases for the $N_i/N_{i+1}$ in a model $N_B^*$ of $Th(M_A^*)$  where the $N_i$ are defined in $N^*_B$ with the same formulas that define the $M_i$  in $M^*_A$.  
\end{proof}
Assume $A$ and $M$ satisfy the usual conditions, $\bar{c}=(c_1, \ldots , c_n)$ is a $\Z$-pseudo-basis of $A$ of period $\bar{f}=(f_1, \ldots f_n)$ and $\bar{u}=(u_1, \ldots , u_s)$ is a $\Z$-pseudo-basis of $M$ of period $\bar{e}=(e_1, \ldots , e_s)$. Then
\begin{enumerate}
\item for any $1\leq i,j\leq n$ there exist integers $s_k(c_i,c_j)$ such that $c_ic_j=\sum_{k=1}^ns_k(c_i,c_j)c_k$, 
\item for any $1\leq i \leq n$ and $1\leq j \leq s$ there exist integers $s'_k(c_i,u_j)$ such that $c_iu_j=\sum_{k=1}^ss'_k(c_i,u_j)u_k$,
\item for any $1 \leq i \leq n$ if $f_i< \infty $ then there exist integers $t_k(f_ic_i)$ such that $f_ia_i=\sum_{k=1}^n t_k(f_ic_i)c_i$,
\item for any $1\leq i \leq s$ if $e_i< \infty$ then there exists integers $t'_k(e_iu_i)$ such that $e_iu_i=\sum_{k=1}^s t'_k(e_iu_i)u_i$.
\end{enumerate}
The integers introduced above are called \emph{the structural constants} associated to the pseudo-bases $\bar{a}$ and $\bar{u}$. We assume an arbitrary but fixed ordering on the set of structure constants. It is easy to verify that $\langle M, A \rangle$ is determined up to isomorphism, as a two-sorted module, by the periods $\bar{e}$, $\bar{f}$ and the associated structure constants.

 Finally we are ready to complete the proof of Theorem~\ref{elemmod:thm}. 
 
 \noindent\emph{Proof of Theorem~\ref{elemmod:thm}.} From Proposition~\ref{basedef} we have the formula
$\varphi_{\mathfrak{P},n}$ which defines all $\mathbb{Z}$-pseudo-bases $\bar{c}$ and $\bar{u}$ of periods $\bar{f}$ and $\bar{e}$ for $A$ and $M$, respectively, in
$\langle M,A\rangle$. Again by Proposition~\ref{basedef} the same formula defines in $\langle N,B\rangle$ similar $\Z$-pseudo-bases $\bar{d}$ and $\bar{v}$ of $B$ and $N$. We need only to describe the structural constants associated with the pseudo-bases $\bar{c}$ and $\bar{u}$ for $A$ and $M$ respectively. This can be done by a formula, say $\psi_{A,M}$, of the language $L_2$ because all these constants are
integers and there are only finitely many of them. Obviously this implies that the $\Z$-pseudo-bases $(\bar{u},\bar{c})$ and $(\bar{v},\bar{d})$ are $\Z$-pseudo-bases of  
$\langle M,A\rangle$ and $\langle N, B\rangle$ respectively of the same periods and structure constants. So the theorem follows.

\qed

\noindent\emph{Proof of Theorem~\ref{elem-iso-alg:thm}.} The proof is entirely similar to that of Theorem~\ref{elemmod:thm}. In addition to the structure constants listed in items (1)-(4) above for a two-sorted module we need to describe the structure constants defining the multiplication for the ring $C$. Keeping the same notation as the proof of mentioned theorem and replacing $M$ by $C$ we need to describe the integers $t''_k(u_iu_j)$ and $t'''_k(u_ju_i)$ where $u_iu_j= \sum_{k=1}^s t''_k(u_iu_j)u_k$ and $u_ju_i= \sum_{k=1}^s t'''_k(u_ju_i)u_k$ since $C$ is not necessarily a commutative ring. Again these new structure constants are also integers and could be captured in the first-order theory of $C$. The new structure constants together with the ones from items (1)-(4) above describe $\langle C, A\rangle$ up to isomorphism by a single first-order formula $\phi_{C,A}$ of $Th(\langle C,A\rangle)$. Clearly if an algebra $\langle D, B\rangle$ from $\mathcal{A}$ satisfies $\phi_{C,A}$ then $\langle C, A \rangle\cong \langle D, B\rangle.$   

\qed

\section{Coordinatization of finitely
generated nilpotent groups admitting exponents in a binomial principal ideal domain}\label{coordinatization}
In this section we generalize Philip Hall's theorem (see \cite{hall}, Section 6) on coordinatization of finitely generated torsion-free nilpotent groups to finitely generated $A$-groups, where $A$ is a binomial PID. In the case that $G$ is a torsion-free $A$-group by \emph{$A$-coordinatization} we mean that one can choose a Mal'cev (canonical) basis $u_1, \ldots , u_n$, $u_i\in G$, with
respect to which any element $g\in G$ can be uniquely written as the
product
$$g=u_1^{t_1(g)}\ldots u_m^{t_m(g)}, \quad t_i(g)\in A,$$
where $t_i(g)$ is called the $i$-th coordinate of $g$ and the coordinates of
the product $gh$ are computed by the coordinates of $g$ and $h$ with
the help of some polynomials depending only on the group $G$ and the
chosen basis $\bar{u}=(u_1,\ldots , u_m)$. Even further we shall introduce the notion of a pseudo-basis for an
arbitrary finitely generated group over a binomial PID $A$, which allows
one to introduce $A$-coordinatization in the general case (even when
$G$ has torsion). 

Recall that we consider an $A$-group $G$ as two different structures. One as a pure group, where the corresponding language is denoted by $L$. The other, a two sorted structure $G_A=\langle G, A \rangle $, where the corresponding language is denoted by $L_1$. The ultimate aim, as should be clear by now, is to study to what extent $G_A$ can be recovered from $G$. For example ~\cite{MR4} studies the question for finitely generated nilpotent pro-$p$ groups,
i.e. when $A$ is the ring of $p$-adic integers $\mathbb{Z}_p$.

Throughout this section $A$ is a binomial PID, and
all the groups, subgroups and homomorphisms are $A$-groups,
$A$-subgroups and $A$-homomorphisms, respectively, i.e. we work in the
category of $A$-groups. 

Let $M$ be a finitely generated abelian $A$-group, (i.e. a finitely generated module
over $A$). Then $M$ admits a decomposition $A=Am_1 \oplus \ldots
\oplus Am_n$ into the direct sum of cyclic $A$-groups. Let
$Ann(m_i)=\{\alpha \in A:\alpha m_i=0\}$ be the annihilator ideal of $m_i$
and let $e_i$ be a fixed generator of $Ann(m_i)$. Then $e_i$ will be
called a \textit{period} of $m_i$ and denoted by $e(m_i)$ (if $Ann(m_i)=0$ we
let $e(m_i)= \infty$).
\begin{defn}The tuple $\bar{m}=(m_1,\ldots , m_n)$ of elements of the
abelian $A$-group $M$ is called a pseudo-basis of
period $e(\bar{m})=(e(m_1),\ldots ,e(m_n))$ of $M$.\end{defn}
For the ideal $Ann(m_i)$ we fix an element $d_i(x_0)$ in each coset
$x_0+Ann(m_i)$ and define a function (a section)
$d_i:A\rightarrow A$, which is constant on each coset of $Ann(m_i)$,
i.e. if $x -x_0\in Ann(m_i)$ then $d_i(x)=d_i(x_0)$. In particular if
$e_i=\infty$ then $Ann(m_i)=0$ and so $d_i(x)=x$. The tuple of
sections $d(\bar{m})=(d_1, \ldots , d_m)$ will be called a
\textit{a section of the pseudo-basis $\bar{m}$}.
\begin{rem} The pseudo-basis $\bar{m}$ and the section $d(\bar{m})$ give
an $A$-coordinatization of $M$, because any element $g\in M$ admits the unique
decomposition
$$g=m_1^{t_1(g)}\cdots m_n^{t_n(g)},$$
(in the multiplicative notation), where $t_i(g)=d_i(\alpha_i)$ for
certain elements $\alpha_i \in A$, $i=1, \ldots, n$, defined uniquely
modulo $Ann(m_i)$.\end{rem}
\begin{defn}Let $G$ be an arbitrary finitely generated nilpotent $A$-group. Let
\begin{equation}\label{eq1}G=G_1\geq \ldots \geq G_{c+1}=1\end{equation} be a
certain central series of $A$-subgroups of $G$. Let $\bar{v}_i$ be
the tuple of elements of $G_i$, the images of which form a pseudo-basis of
period $e(\bar{v_i})$ with the section $d(\bar{v}_i)$ of the
finitely generated $A$-group $G_i/G_{i+1}$, $i=1, \ldots , c$. Then the
tuple $\bar{u}=\bar{v}_1\cup \ldots \cup \bar{v}_c$ will be called a
pseudo-basis of $G$ associated with the series~\eqref{eq1}, and the
tuples $e(\bar{u})=e(\bar{v}_1)\cup \ldots \cup e(\bar{v}_c)$ and
$d(\bar{u})=d(\bar{v}_1)\cup \ldots \cup d(\bar{v}_c)$ are the period
and the section of the pseudo-basis $\bar{u}$.\end{defn}
\begin{prop}Let $G$ be a finitely generated $A$-group and $\bar{u}=(u_1,\ldots ,
u_m)$ a pseudo-basis for $G$ of period $e(\bar{u})$ with the section $d(\bar{u})$. Then
\begin{enumerate}
\item any element $g\in G$ can be represented uniquely in the form
\begin{equation}\label{eq2} g=u_1^{t_1(g)}\ldots u_m^{t_m(g)},\end{equation}
where for each $i=1,\ldots , m$, $t_i(g)=d_i(\alpha_i)$ for a certain element $\alpha_i\in A$, defined uniquely modulo the ideal $Ann(u_i)=Ae(u_i)$;
\item The $A$-subgroups $G_i=\langle u_i,\ldots, u_m\rangle_A$ form a central
series $$G=G_1> \ldots > G_m > 1$$ in
$G$.\end{enumerate}\end{prop}
\begin{proof}The proof is by induction on $m$. The case $m=1$ is
trivial. It is clear from the definition of a pseudo-basis that
$G_m=\langle u_m\rangle_A$ lies in the center of $G$ and the images $\bar{u}_1$,
\ldots , $\bar{u}_{m-1}$ form a pseudo-basis of the group $G/G_m$ of
period $(e(u_1), \ldots , e(u_{m-1}))$ with the section $(d_1,
\ldots , d_{m-1})$. By induction $G/G_m$ satisfies  conditions
(1) and (2) of the statement of the proposition. Now one can easily
prove (1) and (2) for the entire group $G$.

\end{proof}
Fix an arbitrary pseudo-basis $\bar{u}=(u_1, \ldots, u_m)$ with the section $d(\bar{u})=(d_1, \ldots, d_m)$ of period $e(\bar{u})=(e_1,
\ldots , e_m)$. The elements $t_i(g)=d_i(\alpha_i)$ from
decomposition~\eqref{eq2} are called \emph{the coordinates of $g$} in the
pseudo-basis $\bar{u}$. We shall also write equation~\eqref{eq2} in the
vector form:
$$g=\bar{u}^{t(g)}, \quad \textrm{where  }t(g)=(t_1(g), \ldots
t_m(g)).$$
\begin{defn} the set of coordinates
$$\Gamma(\bar{u})=\{t_k([u_i,u_j]),t_k(u_i^{e_i}): 1 \leq i,j,k\leq
m\}$$
is called the set of structural constants of the pseudo-basis $\bar{u}$
and the subring $A(\bar{u})=\mathbb{Z}(\Gamma(\bar{u}))\leq A$,
generated over $\mathbb{Z}$ by the set $\Gamma(\bar{u})$ is called the
definition ring of the basis $\bar{u}$. \end{defn}
From now on we shall represent
$\Gamma(\bar{u})$ in an arbitrary but fixed way in the form of a
tuple $\Gamma(\bar{u})=(\gamma_1, \ldots , \gamma_s)$ so that by
the index $i$ one can uniquely determine which coordinate of which
element $\gamma_i$ is.

For any section $d_i:A\rightarrow A$ and for a fixed period
$e_i$ ($e_i$ is a generator of $Ann(u_i)$) we define the function
$r_i:A\rightarrow A$ from the equality $x=e_ir_i(x)+d_i(x)$. Since $A$
has no zero-divisors $r_i(x)$ is defined uniquely.

Let $R_i$, $D_i$, $B_i$, $i=1,\ldots m$, be unary function
symbols, $+$ and $\cdot$ be binary function symbols and $c_n$ and
$b_{\gamma}$ be constant symbols for $n\in \mathbb{N}$ and $\gamma\in
\Gamma(\bar{u})$. Consider the signature
$$\mathfrak{F}_m=\langle R_i, D_i, B_n, + , ., c_n, b_{\gamma}:n\in
\mathbb{N}, \gamma\in \Gamma(\bar{u}), 1\leq i \leq m \rangle.$$
Denote by $\mathfrak{TF}_m$ the set of all possible terms over
$\mathfrak{F}_m$.
\begin{defn}We shall say that a term $t(x_1,\ldots , x_n) \in
\mathfrak{TF}_m$ computes a function $f:A^n\rightarrow A$ with respect
to a pseudo-basis $\bar{u}$ of period $e(\bar{u})$ with the
section $d(\bar{u})$ and the set of structural constants $\Gamma
(\bar{u})$ if after the substitutions: $R_i(x)\rightarrow r_i(x)$,
$D_i(x)\rightarrow d_i(x)$, $B_n(x)\rightarrow \binom{x}{n}$, $x+y
\rightarrow x+y$, $xy \rightarrow xy$, $c_n\rightarrow n$ and
$b_{\gamma}\rightarrow \gamma$ the obtained operation $t^u(x_1, \ldots
, x_n):A^n \rightarrow A$ is equal to $f$.\end{defn}
\begin{thm}[Uniform Coordinatization]\label{UC}
For any natural number $m$ the following statements are true.
\begin{enumerate}
\item There exist terms $M_1, \ldots , M_m \in \mathfrak{TF}_m$
computing respectively the multiplication coordinate functions $t_1(gh)$, \ldots
, $t_m(gh)$ with  respect to any pseudo-basis $\bar{u}=(u_1, \ldots ,
u_m)$ of any finitely generated $A$-group over any binomial PID, $A$.
\item There exist terms $D_1, \ldots , D_m \in \mathfrak{TF}_m$,
computing respectively the exponentiation coordinate functions
$t_1(g^{\alpha})$, \ldots , $t_m(g^\alpha)$ with respect to any
pseudo-basis $\bar{u}=(u_1,\ldots, u_m)$ of any finitely generated $A$-group $G$ over
any binomial PID, $A$.\end{enumerate}\end{thm}
Before we proceed to the proof of the theorem  we state a
definition. Let $W(x_1, \ldots , x_m, y_1 , \ldots, y_m, z_1, \ldots
z_m)=W(\bar{x},\bar{y},\bar{z})$ be an arbitrary group word. Let
$g=W(u_1^{\alpha_1},\ldots, u_m^{\alpha_m},u_1^{\beta_1}, \ldots,
u_m^{\beta_m},u_1, \ldots ,u_m)=W(\bar{u}^{\alpha},\bar{u}^{\beta},
\bar{u})$ be an element of a $A$-group $G$ with the pseudo-basis
$\bar{u}$. It is clear that the coordinates $t_i(g)$ in the basis
$\bar{u}$ are functions $f_i(\alpha,\beta)$ of $\alpha=(\alpha_1,
\ldots, \alpha_m)$ and $\beta=(\beta_1, \ldots , \beta_m)$. We shall
say that coordinates of $g$ in the basis $\bar{u}$ are computed
$m$-\textit{universally}, and the word $W$ itself is
$m$-\textit{universal}, if there exist terms $T_{1,W}, \ldots, T_{m,W}$
computing the functions $f_1(\alpha,\beta), \ldots , f_m(\alpha,
\beta)$ with respect to any pseudo-basis $\bar{u}=(u_1,\ldots u_m)$ in
any $A$-group over any binomial PID $A$. In part (1) of the theorem it
is demanded to prove that the word $W_0=x_1\ldots x_my_1\ldots y_m$ is
$m$-universal. Similarly part (2) of the theorem requires a proof of
$m$-universality of $W_1=(x_1\cdots x_m)^{\lambda}$, $\lambda$ being a
variable for elements of the ring $A$ (here $W_1$ is a word in the
two-sorted language $L_1$).
\begin{lem}The following statements are true:
\begin{enumerate}
\item if the word $W_0(\bar{x},\bar{y})$ is $m$-universal then any
group word $W(\bar{x}_1,\ldots ,\bar{x}_n)$ of the language $L$ is $m$-universal,
\item if the words $W_0(\bar{x},\bar{y})$ and $W_1(\bar{x},\lambda)$
are $m$-universal then any word $$W(\bar{x}_1,\ldots , \bar{x}_n, \lambda_1, \ldots, \lambda_k)$$ of the language $L_1$ is $m$-universal.
\end{enumerate}\end{lem}
\begin{proof}The proof is by induction on the number of letters in the
word $W$.

\end{proof}

  \begin{proof}[Proof of Theorem~\ref{UC}]
We proceed with induction on the length of the pseudo-basis $\bar{u}=(u_1, \ldots , u_m)$. Let $\alpha=(\alpha_1,\ldots , \alpha_m)$ and $\beta=(\beta_1,\ldots , \beta_m)$ be the coordinates of the elements $x$ and $y$:
$$x=u_1^{\alpha_1} \cdots u_m^{\alpha_m},\quad y=u_1^{\beta_1}\cdots u_m^{\beta_m},$$
in vector notation $x=\bar{u}^{\alpha}$, $y=\bar{u}^{\beta}$. Then for $m=1$ we obtain
$$t_1(xy)=d_1(\alpha_1+\beta_1), \quad t_1(x^{\alpha})=d_1(\alpha_1\alpha).$$
Now let $\bar{u}=u_1 \cup \bar{v}$, where $\bar{v}=(u_2, \ldots , u_m)$. We have
\begin{equation*}\begin{split}
xy=&u_1^{\alpha_1}\ldots u_m^{\alpha_m}u_1^{\beta_1}\ldots u_m^{\beta_m}\\
&=u_1^{\alpha_1+\beta_1}(u_1^{-\beta_1}u_2u_1^{\beta_1})^{\alpha_2}\ldots (u_1^{-\beta_1}u_mu_1^{\beta_1})^{\alpha_m}u_2^{\beta_2}\ldots u_m^{\beta_m}\\
&= u_1^{\alpha_1+\beta_1}\prod_{i=2}^m(u_1^{-\beta_1}u_iu_1^{\beta_1})^{\alpha_i}\prod_{i=2}^mu_i^{\beta_i}\end{split}\end{equation*}
In the case $e_1\neq \infty$ it is necessary to rewrite  the first factor in the following way:
$$u_1^{\alpha_1+\beta_1}=u_1^{d_1(\alpha_1+\beta_1)}(u_1^{e_1})^{r_1(\alpha_1+\beta_1)}=u_1^{d_1(\alpha_1+\beta_1)}(\bar{v}^{\gamma})^{r_1(\alpha_1+\beta_1)}$$
where $\gamma$ is the tuple of structural constants associated to $u_1^{e_1}$ from $\Gamma(\bar{u})$.
For proving the theorem it suffices to verify the $m$-universality of
factors of the type $(\bar{v}^{\gamma})^{r_1(\alpha_1+\beta_1)}$ and
$u_1^{-\beta_1}u_iu_1^{\beta_1}$. For the first one this is
true by induction, since $|\bar{v}|=m-1$, $\gamma$ is just a tuple of structure constants computable by terms of $\mathfrak{TF}_m$ and the function $r_1(\alpha_1+\beta_1)$
is computed by a term from $\mathfrak{TF}_m$.

It remains to consider the case of $u_1^{-\beta}u_iu_1^{\beta_1}$,
$i\geq 2$. For any $x \in A$ we have,
$$u_1^{-x}u_iu_1^x=u_1^{-x}u_iu_1^xu_i^{-1}u_i=u_1^{-x}(u_iu_1u_i^{-1})^xu_i.$$
Then,
\begin{equation*}\begin{split}
(u_iu_1u_i^{-1})^x &=(u_1[u_1,u_i^{-1}])^x\\
&= u_1^x[u_1,u_i^{-1}]^x\tau_m^{-\binom{x}{m}}\cdots \tau_2^{-\binom{x}{2}},\end{split}\end{equation*}
by the Hall-Petresco formula, where $\tau_k$ is a fixed product of
simple commutators of length no less than $k$ of elements $u_1$ and
$[u_1,u_i^{-1}]$. Consequently,
$$u_1^{-x}u_iu_1^x=[u_1,u_i^{-1}]^x \tau_m^{-\binom{x}{m}}\cdots
\tau_2^{-\binom{x}{2}}u_i.$$
Now it is sufficient to prove that the elements $[u_1,u_i^{-1}]$,
$\tau_1,\ldots , \tau_m$ are $m$-universal. We have
$$[u_1,u_i^{-1}]=u_i[u_i,u_1]u_i^{-1}=u_i\bar{v}^{\delta}u_i^{-1},$$
where $\delta$ is a tuple of structural constants from
$\Gamma(\bar{u})$. By induction the word $u_i\bar{v}^{\delta}u_i^{-1}$
is $(m-1)$-universal and hence $[u_1,u_i^{-1}]$ is $m$-universal.

Since $\tau_i$ is a fixed product of simple commutators of the
elements $u_1$ and $[u_i,u_1]$ it suffices to show a way of uniform
elimination of all occurrences of $u_1$ into an arbitrary commutator
$[g_1,\ldots , g_n]$, $g_i$ being either $u_1$ or a fixed
$(m-1)$-universal element from $\langle u_2, \ldots, u_m\rangle
_A$. For this purpose it suffices to prove that if
$g=\bar{v}^{t(g)}$ then the coordinates of the commutator $[g,u_1]$
(as functions of $t(g)$) are computed by some terms from
$\mathfrak{TF}_m$. Let the decomposition of $g$ by the base $\bar{v}$
begin with $u_i^x$, $i\geq 2$, i.e. $g=u_i^xf$. Then
$$[g,u_1]=[u_i^xf,u_1]=[u_i^x,u_1][[u_i^x,u_1],f][f,u_1].$$
Then
\begin{equation}\begin{split}\label{eq3}
[u_i^x,u_1]&= u_i^{-x}u_1^{-1}u_i^{x}u_1\\
&=u_i^{-x}(u_1^{-1}u_iu_1)^x\\
&=[u_i,u_1]^x\tau_2(u_i^{-1},u_1^{-1}u_iu_1)^{\binom{x}{2}}\cdots
\tau_m(u_i^{-1},u_1^{-1}u_iu_1)^{\binom{x}{m}}.\end{split}\end{equation}
The coordinates of the elements $[u_i,u_1]$ and
$u_1^{-1}u_iu_1=u_i[u_i,u_1]$ are computed by the terms from
$\mathfrak{TF}_m$, then by induction (everything is computed in the
basis $\bar{u}$) the coordinates of each factor in the product
appeared in the last line of~\eqref{eq3}, and hence the coordinates
of the whole product are computed by terms from $\mathfrak{TF}_m$. Now
it remains to notice that the decomposition of $[f,u_1]$ is obtained
by induction. As the final result we obtain the decomposition of $xy$
into the product of $m$-universal elements lying in the group
$G_2=\langle u_2, \ldots , u_m\rangle_A$. By induction their product
is computed in the base $\bar{v}$ $(m-1)$-universally, which proves
part (1) of the statement of the theorem.

Now consider exponentiation:
$$(u_1^{\alpha_1}\ldots
u_m^{\alpha_m})^{\lambda}=u_1^{\alpha_1\lambda} \ldots
u_m^{\alpha_m\lambda}\tau_m^{-(^{\lambda}_m)}\ldots
\tau_2^{-(^{\lambda}_2)},$$
$\tau_i$ being a product of fixed commutators of weight $\geq i$ of
the elements $u_1^{\alpha_1},\ldots , u_m^{\alpha_m}$. According to
part 1, already proved, the elements $\tau_1$, \ldots ,
$\tau_m$ are $m$-universal and lie in the group $G_2$. Then by
induction their powers are $(m-1)$-universal and again according to
part 1 the entire product is $m$-universal, which proves part 2 of
the theorem. The theorem is proved.

\end{proof}
As a corollary of the theorem we immediately obtain the following
theorem.
\begin{thm}Two finitely generated $A$-groups $G$ and $H$ are $A$-isomorphic if and
only if they have pseudo-bases of the same period with the same set of
structural constants.\end{thm}
\begin{proof}According to the theorem above $G$ and $H$ have
pseudo-bases $\bar{u}$ and $\bar{v}$ of the same period with the same
multiplication coordinate functions. Consequently the mapping:
$$u_1^{\alpha_1}\ldots u_m^{\alpha_m}\mapsto v_1^{\alpha_1}\ldots
v_m^{\alpha_m}$$
is a $A$-isomorphism of $G$ onto $H$.

\end{proof}

  \section{ Foundations and additions and regular groups}\label{regular:sec}

The aim of this section is to state and prove some basic facts about regular groups. 

\subsection{Foundations and additions}

Let us recall some definitions. Let $A$ be a binomial principal ideal domain and let $G$ be an arbitrary finitely generated nilpotent $A$-group.
For an $A$-subgroup $N\leq G$ we shall denote by $Is(N)$ the isolator of $N$ in $G$, i.e. an $A$-subgroup
of the type:
$$Is_G(N)=\{x\in G:\exists \alpha \in A\setminus\{0\}~ (x^{\alpha}\in N)\}.$$
When the group $G$ is obvious from the context we drop $G$ from $IS_G(N)$. Define
$$I(G)=Is(G')\cap Z(G).$$
The special gap
\begin{equation}\label{eq4}Z(G)\geq G' \cap Z(G)\end{equation}
will be called \textit{tame}, if $Z(G)=I(G)$.

Let us refine the special gap~\eqref{eq4} with the help of $I(G)$:
$$Z(G)\geq I(G)\geq G'\cap Z(G).$$
The quotient $I(G)/(G'\cap Z(G)$ is a finite direct sum of periodic
cyclic $A$-modules, and the quotient $Z(G)/I(G)$ is a free $A$-module of finite rank. Since $A$ is a principal ideal domain, there exists a free
$A$-module of finite rank $G_0\leq Z(G)$, such that $Z(G)\cong
G_0\times I(G)$.

\begin{defn} Any subgroup $G_0\leq G$ such that
$Z(G)=G_0\times I(G)$ is called an addition of $G$ and the quotient group $G_f=G/G_0$
is called a foundation of $G$, associated with the addition
$G_0$.\end{defn}
 The facts stated above allow us to formulate the following  proposition
\begin{prop} \label{prop1}Any finitely generated $A$-group $G$ over the principal
ideal domain $A$ possesses an addition and a foundation. Besides,
\begin{enumerate}
\item the addition $G_0$ is a free $A$-module of finite rank,
\item the foundation $G_f$ has a tame special gap.
\end{enumerate}\end{prop}
\begin{proof} Part (1) is the direct corollary of the definition of $G_0$ and the structure theorem for finitely generated modules over PID's. For part (2) we firstly claim that 
\begin{equation}\label{zfound:eqn}Z(G/G_0)=Z(G)/G_0.\end{equation} The inclusion $\geq$ is obvious. If $xG_0 \in Z(G/G_0)$ then for all $y\in G$, $[x,y]\in G_0$. But $[x,y]\in G'$ too, so by definition for all $y\in G$, $[x,y]\in G'\cap G_0=1$. So $x\in Z(G)$ as claimed. On the other hand an easy calculation shows that 
\begin{equation}\label{Isfound:eqn}Is((G/G_0)')=Is(G'\cdot G_0)/G_0=Is(G'\cdot Z(G))/G_0\end{equation}
Now \eqref{zfound:eqn} and \eqref{Isfound:eqn} imply that $Z(G/G_0)\leq Is((G/G_0)')$ which is the desired result.\end{proof} 
\subsection{Regular groups}
\begin{defn}A nilpotent $A$-group is called regular, if $G=H\times
G_0$, where $G_0$ is an addition of $G$ and $Is(H')\geq Z(H)$.\end{defn}
If $A$ is a field then by \cite{MR3}, Theorem 7 any $A$-group is
regular. This fact will also follow from Proposition~\ref{rproperties:prop} below. In the case of $A=\mathbb{Z}$ it is not true as the following example shows.
\begin{exmp} let $G$ be a finitely generated group defined in the variety of
2-nilpotent groups by the presentation
$$G=\langle a,b,c:[c^3,a]=1, [c^3,b]=1\rangle.$$
Then  $G_0=\langle c^3\rangle \cong \mathbb{Z}^+$, $G_f=\langle
a,b,c:c^3=1\rangle$ and it is clear that $G\ncong G_f\times G_0$, therefore (see the proposition below)
$G$ is not regular.\end{exmp}
\begin{prop}\label{rproperties:prop} Let $G$ be a finitely generated group over the principal ideal domain $A$. Then the following conditions are
equivalent.
\begin{enumerate}
\item $G$ is a regular group.
\item $G\cong G_f\times G_0$ for a certain addition $G_0$ and the associated foundation $G_f$.
\item $G\cong G_f\times G_0$ for any addition $G_0$ and the associated foundation
$G_f$.
\item $Is(Z(G)\cdot G')=Is(G')\cdot Z(G)$. \end{enumerate}\end{prop}
\begin{proof}The implication (2)$\Rightarrow$(1) is obvious for if
$G=H\times G_0$ is the decomposition from the definition of a regular
group, then $G_0$ is an addition of $G$ and $H\cong G_f=G/G_0$. The
inverse implication (1)$\Rightarrow$(2) follows from~\ref{prop1}.
Let us prove (2)$\Rightarrow $(4). So we can assume $G=G_f\times G_0$. Since
$Is(G'\cdot Z(G))=Is(G'\cdot G_0)$ and $G'\leq G_f$, then
$$Is(G'\cdot G_0)=Is(G'\times G_0)=Is(G')\times G_0=Is(G')\cdot Z(G).$$
To prove (4)$\Rightarrow $(3) assume $Is(G')\cdot Z(G))=Is(G')\cdot Z(G)$. The $A$-group
 $G/Is(G'\cdot Z(G))$ is a finitely generated abelian $A$-group without $A$-torsion, i.e. a
free $A$-module. Consequently,
$$Ab(G)=G/G'=\pi(Is(G'\cdot Z(G)))\times G_1,$$
where $\pi:G\rightarrow
G/G'$ is the canonical epimorphism and $G_1$ is a certain direct $A$-complement of
$\pi(Is(G'\cdot Z(G)))$ in $Ab(G)$ . In turn, for
an arbitrary addition $G_0$ we have
$$\pi(Is(G'\cdot Z(G)))=\pi(Is(G')\cdot Z(G))=\pi(Is(G'))\times G_0.$$
Hence
$$Ab(G)=G_1\times \pi(Is(G')) \times G_0.$$
Let $H$ be the complete pre-image of the subgroup $G_1\times
\pi(Is(G'))$ in $G$.
Then $G=H\cdot G_0$, $H\cap G_0\leq G' \cap G_0 =1$ and $G'\leq
H$. Consequently $G\cong H\times G_0$ and $H\cong G_f$.
The implication (3)$\Rightarrow $(2) is obvious. The proposition is proved.

\end{proof}
\begin{cor} Any finitely generated 2-nilpotent $A$-torsion-free $A$-group is
regular.\end{cor}
\begin{proof}
Actually if the group $G$ is $A$-torsion-free then $Z(G)$ is an
isolated subgroup of $G$. Moreover $G'\leq Z(G)$. Therefore
$$Is(G'\cdot Z(G))=Is(Z(G))=Z(G)=Is(G')\cdot Z(G),$$
and according to the proposition, $G\cong G_f\times G_0$.

\end{proof}
In the sequel, if $G=H\times G_0$ is a regular group, then we shall call the subgroup $H$ a foundation
of $G$ as well, identifying $H\cong G/G_0$.

\begin{prop}Let $G$ be a regular group. Then:
\begin{enumerate}
\item $G\cong G_f \times G_0$ for any foundation $G_f$ and any addition $G_0$ which maybe not associated with each other;
\item all foundations of $G$ are isomorphic to each other, and so are additions. \end{enumerate}\end{prop}
\begin{proof}To prove (1) let $G_0$ be an arbitrary addition of $G$, and let $H$ be a certain foundation of $G$ which maybe not
associated with $G_0$. It is clear from the definition of an addition
and a foundation that $H \geq Is(G')$ and $Z(G)=(H\cap Z(G)) \times
G_0$. Hence $G=H\cdot Z(G)=H\cdot G_0$ and $H\cap G_0=1$, i.e. $G\cong H\times
G_0$. Statement (2) follows from (1).

\end{proof}

\section{Elementary equivalence of finitely generated nilpotent groups}\label{main:sec}
In this section we give a structural and algebraic characterization for elementary equivalence of finitely generated nilpotent groups.

Let us first put together a proof of Proposition~\ref{max-refined:prop}. We start be proving two lemmas.
\begin{lem} \label{F_R-inf:lem} Assume $G$ is an infinite finitely generated nilpotent group such that $Z(G)\leq Is(G')$ and let $(R)$ be an arbitrary central series for $G$. Then all the abelian groups $R^u$, $R^l$ and $G/V_R$ are infinite. \end{lem}
\begin{proof} Assume $R^l$ is finite. Then by definition of $R^l$, $G'$ is finite. This implies that $Is(G')$ is finite. Since $Z(G)\leq Is(G')$, $Z(G)$ is finite. So by a standard argument, say as in Corollary 2.3 of \cite{war}, $G$ has to have finite exponent. Now $G$ being finitely  generated and nilpotent we can conclude that $G$ is finite, contradiction. So $R^l$ has to be infinite. Recall that $F_R:G/V_R\times R^u \to R^l $ is a full bilinear map. So the induced homomorphism $\bar{F}_R: (G/V_R)\otimes R^u \to R^l$, where $\otimes$ is the tensor product of abelian groups, is onto. But if one of $G/V_R$ or $R^u$ was finite then the $(G/V_R)\otimes R^u$ would be finite, which is impossible.\end{proof}
\begin{cor} \label{P(F_R)-inf:cor} Under the assumptions of Lemma~\ref{F_R-inf:lem} all the rings $P(F_R)$, $P_R$ and $A_R$ include a copy of the ring of integers $\Z$.\end{cor}
\begin{proof} By lemma~\ref{F_R-inf:lem} all the finitely generated abelian groups $G/V_R$ and $R^u$ and $R^l$ are infinite. So they are all exact $\Z$-modules, while obviously the map $F_R$ is $\Z$-bilinear. So by definition $\Z\cdot 1 \leq P(F_R)$. Now that $P(F_R)$ includes a copy of $\Z$ the subrings $P_R$ and $A_R$ of $P(F_R)$ also include a copy of $\Z$ by definition.\end{proof}      

\noindent \emph{Proof of Proposition~\ref{max-refined:prop}.} Assume $G_0$ is any addition of $G$ and $G_f$ the corresponding foundation. We observe that the foundation $G_f\cong G/G_0$ is infinite otherwise the group would be abelian-by-finite. Let $\pi:G\to G_f$ denote the canonical epimorphism. Note that $R_{i}^l/R_{i+1}^l\cong \pi(R_i^l)/\pi(R^l_{i+1})$ for all major direct factors of $R^l$. The same for major direct factors of $R^u$. Since $G_0\leq V_R$ we also have $G/V_R\cong \pi(G)/V_{\pi(R)}$. Since $G_f$ satisfies hypothesis of Lemma~\ref{F_R-inf:lem} we conclude that $R^l$, $R^u$ and $G/V_R$ are infinite. This clearly implies the first part of the statement. Moreover by Corollary~\ref{P(F_R)-inf:cor} (or rather its proof) the ring $A=A_R(G)$ contains a copy of $\Z$. By Proposition~\ref{prop64} one can interpret in $G$ the ring $A=A_R(G)$ and its action on each quotient of the consecutive terms of $(U(R))$ and $(L(R))$ except the quotients associated to the special gaps. More specifically if $U_i>U_{i+1}$ are consecutive terms of $(U(R))$ avoiding the special gap, then the two-sorted module $\langle \bar{U}_i, A \rangle$, where $\bar{U}_i=U_i/U_{i+1}$ is interpretable in $G$. The same for the consecutive terms of $(L(R))$ which avoid the corresponding special gap. The interpretations are uniform with respect to $Th(G)$. 

By observations above and Proposition~\ref{basedef} for each $i$ as above there is a sequence of definable $A$-submodules 
$$\bar{U}_{i}=\bar{U}_{i1}>\bar{U}_{i2}> \ldots >\bar{U}_{in_{i}}>\bar{U}_{i,n_{i}+1}=0,$$
of $\bar{U}_i$ where either $\bar{U}_{ij}/\bar{U}_{i,{j+1}}$, $1\leq j \leq n_i$, is finite or it is infinite and the $\Z$-module structure of the quotient is interpretable in its $A$-module structure, that is, $\langle \bar{U}_{ij}/\bar{U}_{i,{j+1}}, \Z, s_{ij}[\Z]\rangle$, where $s_{ij}[\Z]$ describes the action of $\Z$, is interpretable in $\langle \bar{U}_{ij}/\bar{U}_{i,{j+1}}, A, s_{ij}\rangle$. This time interpretations will be uniform with respect to finitely generated models of $Th(G)$. This concludes the proof.

\qed

We still need a version of Proposition~\ref{max-refined:prop} in case the (refined) special gaps $Is(G')\cdot Z(G) > Is(G')$ and  $Is(G')>Z(G)\cap Is(G')$ exist. To that end we proceed by proving a few lemmas.
\begin{lem}Assume $G$ is a finitely generated nilpotent group then there is a formula of $L$ that defines $Is(G')$ uniformly with respect to $Th(G)$.\end{lem}
\begin{proof}By Proposition~\ref{prop62} $G'$ is definable in $G$ uniformly with respect to $Th(G)$ by a formula $\phi$ of $L$. Since $Is(G')/G'$ is a a finite group, $g\in Is(G')$ if and only if $g^m\in G'$, where $m$ is the exponent of the quotient $Is(G')/G'$. So if $\phi(x)$ defines $G'$ in $G$, then $\phi(x^m)$ defines $Is(G')$ in $G$. Now assume $H\equiv G$ and set
$$H_1=\{x\in H: H\models \phi(x^m)\}.$$
Obviously $H_1\subseteq Is(H')$. Pick $x\in H$ so that $x^n\in H'$ for some $n\in \N\setminus \{0\}$. Note that $$G\models \forall x (\phi(x^n)\rightarrow \phi(x^m)).$$
So the same sentence is true in $H$ implying that $H_1=Is(H')$ which concludes the proof.

\end{proof}

\begin{cor}\label{bigseries}Assume $G\equiv H$ are finitely generated nilpotent groups. Then $G_0\cong H_0$ for any respective additions of $G$ and $H$.\end{cor}
\begin{proof} It follows from lemma above and uniform definability of $Z(G)$ that $I(G)$ is uniformly definable in $G$. So the quotient $Z(G)/I(G)\cong G_0$ is uniformly interpretable in $G$. Thus we conclude that $G_0\cong Z(G)/I(G)\equiv Z(H)/I(H)\cong H_0$. But $G_0$ and $H_0$ are free abelian groups of finite rank, hence $G_0\cong H_0$.

\end{proof}
\begin{cor} \label{gq}The subgroups $M(G)=Is(G'\cdot Z(G))$ and $N(G)=Is(G')\cdot Z(G)$ are uniformly definable in $G$. Therefore the finite abelian quotient
$$M(G)/N(G)$$ is interpretable in $G$ uniformly with respect to $Th(G)$.\end{cor}
\begin{lem}There exists a central series
 \begin{equation}\label{fcs}\begin{split} G=G_1 > G_2 &> \ldots > G_p > Is(G'\cdot Z(G))> Is(G')\cdot Z(G) > Is(G')\\
&> G'> \ldots >G_t >G_{t+1}=1,\end{split}\end{equation}
such that each quotient of the consecutive terms of the series except possibly the quotient $Is(G')\cdot Z(G) > Is(G')$
is either finite or the action of $\Z$ on that quotient is interpretable in $G$.\end{lem}
\begin{proof} Consider the lower central series
$$G=\Gamma_1(G)>\Gamma_2(G)>\ldots >\Gamma_c(G)>\Gamma_{c+1}=1, \quad (\Gamma)$$
of $G$. Recall the completed series $(L(\Gamma))$ associated to $(\Gamma)$ (see Section~\ref{cs}). Then by Proposition~\ref{prop64} there is a ring $A_\Gamma$
whose action on each quotient of $(L(\Gamma))$, except possibly on the quotient of the gap $G'\cdot Z(G)> G'$, is interpretable in $G$ uniformly. By Proposition~\ref{ringf} the ring $A_\Gamma$ is a commutative associative $\Z$-algebra with unit which has a finitely generated additive group. So each quotient of the $(L(\Gamma))$, except the one corresponding to the gap, is a finitely generated $A_\Gamma$-module. So by Proposition~\ref{basedef} and its proof there exists a refinement $(L'(\Gamma))$ so that
\begin{itemize}
\item each term of $(L'(\Gamma))$ is definable in $G$ uniformly with respect to $Th(G)$, and
\item each corresponding quotient (except the quotient corresponding to the gap) is either finite or is infinite and the action of $\Z$ on it is interpretable in $G$ uniformly with respect to $Th(G)$.\end{itemize}
 Since the quotients $M(G)/N(G)$ and $Is(G')/G'$ are finite and appropriate terms of $(L'(\Gamma))$ can be used to fill in the gap $G'> 1$ we only need to find a refinement of $G>Is(G'\cdot Z(G))$ satisfying the condition. Again by Proposition~\ref{basedef} there is a series
 $$G=G_1 > G_2 >\ldots > G_l > Z(G)\cdot G'=G_{i_0+1},$$
for some $i_0\in \N$ each quotient of which has the desired property. So consider the series
$$G=G_1 > G_2\cdot\isg> G_3\cdot \isg > \ldots > G_{i_0+1}\cdot \isg.$$
For each $i$ consider the obvious epimorphism
$$\phi_i: G_i/G_{i+1}\rightarrow (G_i\cdot\isg)/ (G_{i+1}\cdot \isg).$$
Each $\phi_i$ is interpretable in $G$ and thus can be used in an obvious manner to interpret the action of $\Z$ on $(G_i\cdot\isg)/ (G_{i+1}\cdot \isg)$. This finishes the proof.

\end{proof}

\begin{rem} In the following statements we assume $ Z(G)> Is(G')\cap Z(G) $ and $M(G)/N(G)\neq 1$, just to make notation less complicated and statements shorter. It will be clear from discussions that analysis of the ``tame special gap case'' is actually much easier and follows by making proper modifications in the arguments that follow.\end{rem}

\begin{lem}\label{sp1}Consider the series~\eqref{fcs} obtained in Lemma~\ref{bigseries}.  Then there exists a pseudo-basis $\bar{u}$ of $G$ of length $m$ adapted
to~\eqref{fcs} and there are natural numbers $1< i_0<i_1<i_2<m$ so that:
\begin{enumerate}
\item $(u_1M(G), u_2M(G), \ldots, u_{i_0}M(G)$ is pseudo-basis of $G/M(G)$,
\item $(u_{i_0+1}N(G), u_{i_0+2}N(G), \ldots, u_{i_1}N(G))$ is pseudo-basis of $M(G)/N(G)$,
 \item $(u_{i_1+1}Is(G'), u_{i_1+2}Is(G'),\ldots , u_{i_2}Is(G'))$ is a pseudo-basis of $N(G)/Is(G')$,
\item $(u_{i_2+1}, u_{i_2+2}, \ldots , u_m)$ is a pseudo-basis of $Is(G')$.
\end{enumerate}		
Moreover if $K_i=\langle u_i, \ldots , u_m\rangle$, then the series
$$G=K_1>K_2>\ldots >K_m>1 \qquad (K),$$
is a central series and unless $i_1+1< i \leq i_2$ each subgroup $K_i$ is definable in $G$ with constants $(u_i, \ldots , u_m)$. \end{lem}
\begin{proof}Most of the statement is clear. We shall just comment on definability of the $K_i$, $i\neq i_1+2, \ldots , i_2$. To show this we proceed by induction on $i$. Indeed either $K_m$ is infinite cyclic and $K_m=u_m^{\Z}$ where the action of $\Z$ on $u_m$ is interpretable in $G$ or it is finite cyclic. So assume $i> i_2$ and assume the statement is true for $i+1$. Then
$$K_i=\{xy: \exists a\in \Z (xK_{i+1}= (u_iK_{i+1})^a \wedge y\in K_{i+1})\},$$
where the right hand side describes a definable subset by induction hypothesis and the fact that $\langle u_iK_{i+1}\rangle$ is interpretable in $G$ as it admits an interpretable action of $\Z$. If $i=i_1+1$ then $K_{i}=Is(G')\cdot Z(G)$ which is definable uniformly by Corollary~\ref{gq} and the result follows by an induction on $i$ for $i\leq i_1$. 

\end{proof}
\begin{lem} \label{mainlem} Given a pseudo-basis $\bar{u}$ of $G$ of length $m$ and $K_i\leq G$ as in Lemma~\ref{sp1} and Lemma~\ref{stb} assume that $e_i$ are the periods of the $u_i$ and $t_k([u_i, u_j])$ and $t_k(u_i^{e_i})$ are the structure constants associated to $\bar{u}$.  Assume also $n=i_1-i_0$ and $p=i_2-i_1$. Then there exists a first order formula $\Phi(x_1, x_2, \ldots, x_m)$ in $L$ so that
$$G\models \Phi(\bar{u}),$$
and $\Phi(\bar{u})$ expresses that:
\begin{enumerate}
\item for all $x\in G$ there exists a unique tuple
$$(a_1,a_2, \ldots a_{i_1}, a_{i_2+1}, a_{i_2+2},\ldots , a_m)\in \Z^{m-p},$$
and a tuple $$(g_1,g_2, \ldots, g_{i_1}, g_{i_2+1}, g_{i_2+2},\ldots , g_m)\in G^{m-p},$$
and an element $w_0 \in (Z(G) \setminus Is(G'))\cup {1}$ unique modulo $Is(G')\cap Z(G)$,
so that $$x= \prod_{\substack{1\leq i \leq m \\
i\neq i_1+1, \ldots, i_2}}g_iw_0,$$
 and $g_iK_{i+1}=(u_iK_{i+1})^{a_i}$,
\item for all $1\leq i,j \leq m$
$$[u_i, u_j]=\prod_{k=i_2+1}^m u_k^{t_k([u_i,u_j])},$$
\item  for all $i$ except $i=i_0+1, \ldots , i_1$, if $e_i\neq \infty$ then  $$u_i^{e_i}=\prod_{k=i_2+1}^m u_k^{t_k(u_i^{e_i})},$$
\item for each $i=i_0+1, \ldots , i_1,$
 $$u_i^{e_i}=\prod_{k=i_1+1}^m u_k^{t_k(u_i^{e_i})},$$
\item the set $\{u_{i_1+1}Is(G'), u_{i_1+2}Is(G'), \ldots , u_{i_2}Is(G')\}$ is a maximal linearly independent subset over $\Z$ of $(Is(G')\cdot Z(G))/Is(G')$ and therefore the rank of the addition $G_0$ containing these $u_i$'s (and therefore the rank of any addition) is $p=i_2-i_1$.
\end{enumerate}\end{lem}
\begin{proof}
Pick a pseudo-basis $\bar{u}$ like the one in Lemma~\ref{sp1}. Then (1) is true for $\bar{u}$ by Lemma~\ref{sp1} and uniform definability of $Z(G)$ and $Is(G')$. Statements (2), (3) and (4) are clearly true for $\bar{u}$ by choice and are first-order in an extension $(L,\bar{u})$ of $L$. Note that $t_k$ and $e_i$ are fixed known integers. Also note that
\begin{itemize}
\item $t_k([u_i,u_j])=0$ for any $i$ and $j$, if $k\leq i_2$,
\item $t_k(u_i^{e_i})=0$ if $k=i_1+1, \ldots ,  i_2$ for any $i$ except $i=i_0+1, \ldots , i_1$. \end{itemize}
To express (5) in $L$ firstly recall that $G_0\cong Is(G'\cdot Z(G))/Is(G')=P$ where the right hand side is interpretable in $G$ uniformly. So given any integer $e\geq 2$ we can write a sentence of $L$ that says that the image of $$\{u_{i_1+1}, u_{i_1+2}, \ldots , u_{i_2}\}$$
in $P/eP$ is a basis of $P/eP$, which implies (5).

\end{proof}
Indeed what Lemma~\ref{mainlem} states is that the length and periods associated to the pseudo-basis $\bar{u}$ are captured uniformly in $Th(G)$. Therefore all the structure constants of $\bar{u}$ are captured in $Th(G)$ except those $t_k(u_i^{e_i})$ where $i_0+1\leq i \leq i_1$ and $i_1+1\leq k \leq i_2$. This is mainly due to our failure in expressing in $L$ that $\{u_{i_1+1}, u_{i_1+2}, \ldots , u_{i_2}\}$ is a generating set for $G_0$ despite our success in expressing that this set is a maximal linearly independent set.

Let us make the following agreement stated as a lemma on the choice of elements $u_i$, $i_0+1\leq i \leq i_2$.   
\begin{lem}\label{stb} One can choose the pseudo-basis $\bar{u}$ of $G$ so that
for $i_0+1 \leq i \leq i_1$ and $i_1+1 \leq k \leq i_2$, 
\begin{itemize}
\item $t_k(u_i^{e_i})=1$ if $k= i+n$,
\item $t_k(u_i^{e_i})=0$ if $k\neq i+n$\end{itemize}
where $n=i_1-i_0$.\end{lem}
\begin{proof}
Indeed, the abelian group $(Is(G')\cdot Z(G))/Is(G')$ is a finite index subgroup of the free abelian group $Is(G'\cdot Z(G))/Is(G')$. So by structure theory of finitely generated abelian groups one can choose a basis $w_1Is(G'), \ldots, w_pIs(G')$, of  $Is(G'\cdot Z(G))/Is(G')$, where $p=i_2-i_1$, so that $w_1^{e_{i_0+1}}Is(G'), \ldots, w_p^{e_{i_0+p}}Is(G')$ is a basis of $(Is(G')\cdot Z(G))/Is(G')$. Moreover to make the number $i_1-i_0=n$ minimal we can make the choice so that we get the invariant factor decomposition:
$$\frac{Is(G'\cdot Z(G))}{Is(G')\cdot Z(G)}\cong \frac{\Z}{e_{i_0+1}\Z}\oplus\cdots\oplus \frac{\Z}{e_{i_1}\Z},$$
where $e_{i_0+1}|e_{i_0+2}|\ldots |e_{i_1},$ and for each $i$, $|e_i|>1$.

Now set:
\begin{itemize}
\item $u_{i_0+j}=w_j$ for $1\leq j \leq n$,
\item $u_{i_1+j}=w_j^{e_{i_0+j}}$ for $1\leq j \leq n$, and
\item $u_{i_1+j}=w_j$ for $n+1\leq j\leq p$ \end{itemize}
to get the desired result.

\end{proof}

\begin{prop}\label{maincor} Assume $G\equiv H$ are finitely generated nilpotent groups and $\bar{u}$ is the pseudo-basis of $G$ appearing in Lemma~\ref{mainlem} and Lemma~\ref{stb}. Then there exists a polycyclic central series $(L)$ of $H$
\begin{align*} H=L_1 &> \ldots >L_{i_0+1}=Is(H'\cdot Z(H))>L_{i_0+2}>\ldots>L_{i_1+1}=Is(H')\cdot Z(H)\\
&> L_{i_1+2} > \ldots >L_{i_2+1}=Is(H')> L_{i_2+2}>\ldots >L_m>1,\end{align*}
where each term $L_i$, except possibly when $i_1+1<i\leq i_2$, is definable in $H$ using the same formula that defines $K_i$ in $G$. Moreover
there exists a pseudo-basis $\bar{v}$ of  $H$ of length $m$ adapted to the series above such that
  \begin{enumerate}
  \item for all $1\leq i \leq m$
$$[v_i, v_j]=\prod_{k=i_2+1}^m v_k^{t_k([u_i,u_j])},$$
\item for all $i$ except $i=i_0+1, \ldots , i_1$, if $e_i\neq \infty$ then
$$v_i^{e_i}=\prod_{k=i_2+1}^m v_k^{t_k(u_i^{e_i})},$$
\item there exists a basis $\{v_iIs(H'): i_1+1\leq i \leq i_2\}$ of $Is(H')\cdot Z(H)/Is(H')$, integers $d_k$, $i_1+1\leq k \leq i_1+n$, and an $n\times n$ matrix of integers $c_{ik}$, $i_0+1\leq i\leq i_1$, $i_1+1\leq k \leq i_1+n$ such that

\begin{enumerate}
\item $\displaystyle v_i^{e_i}=\left(\prod_{k=i_1+1}^{i_1+n} v_k^{d_kc_{ik}}\right)\left(\prod_{k=i_2+1}^m v_k^{t_k(u_i^{e_i})}\right),$ for each $i_0+1\leq i \leq i_1$,
\item $|det(c_{ik})|=1$,
\item gcd$(d,e)=1$ where $d=d_{i_1+1}\cdots d_{i_1+n}$ and $e= e_{i_0+1}\cdots e_{i_1}$.  

\end{enumerate}\end{enumerate}\end{prop}
\begin{proof} Let $\Phi$ be the formula obtained in Lemma~\ref{mainlem}. Then $H\models \exists \bar{x}~\Phi(\bar{x})$ and so we may pick a tuple
$$(v_1,v_2, \ldots, v_{i_1}, w_{i_1+1}, \ldots, w_{i_2}, v_{i_2+1}, v_{i_2+2},\ldots , v_m)\in H^{m},$$
that satisfies $\Phi(\bar{x})$. Set $L_i$ as the subgroup defined by the formula defining $K_i$ in $G$. So indeed $L_i=\langle v_i, v_{i+1}, \ldots, v_m\rangle$ if $i> i_2$ and $L_i=\langle v_i, \ldots, v_{i_1}, Is(H')\cdot Z(H) \rangle$ if $i\leq i_1$. The elements $w_{i_1+1}Is(H'), \ldots, w_{i_2}Is(H')$ form a maximal linearly independent subset of the free abelian group $Is(H')\cdot Z(H)/Is(H')\cong H_0$ of rank $p=i_2-i_1$. Considering Lemma~\ref{stb}, part (4) of Lemma~\ref{mainlem} implies for $i_0+1 \leq i \leq i_1$ that:
\begin{equation}\label{viei} v_i^{e_i}= w_{i+n}\prod_{k=i_2+1}^mu_k^{t_k(u_i^{e_i})}.\end{equation}
Again by structure theory of finitely generated abelian groups there exists a basis $\{v_kIs(H'):i_i+1\leq k \leq i_2\}$ for $Is(H')\cdot Z(H)/Is(H')$ and integers 
$d_k$ as in the statement where 
$$H_1=\langle v_i^{d_i}Is(H') : i_0+1\leq i \leq i_0+n\rangle= \langle w_iIs(H') : i_0+1\leq i \leq i_0+n\rangle,$$
as subgroups of $(Is(H')\cdot Z(H))/Is(H')$. Indeed let $H_0$ be any addition of $H$ containing the $w_i$. Then since $Is(H')\cdot H_0=Is(H')\times H_0$, it is always possible to pick the $v_i\in H_0$ in a way that replacing the $w_{i}$ in \eqref{viei} by a linear combination of the $v_i^{d_i}$ does not affect the $t_k(u_i^{e_i})$. 

 Hence there exists a central series as in the statement and a pseudo-basis $\bar{v}$ of $H$ of the same length as $\bar{u}$ and also the same periods. Notice that by our choice of $\bar{v}$ for any $1\leq k,i,j\leq m$, $t_k([v_i,v_j])=t_k([u_i,u_j])$ which verifies (1). Part (2) can be verified in a similar way. Statements 3-(a) and 3-(b) should be clear once it is noted that the matrix $(c_{ik})$ is the change of basis matrix corresponding to the bases obtained for $H_1$ above. To see why 3-(c) is true firstly observe that  the index of $\langle w_iIs(H'): i_1+1\leq i\leq i_2\rangle$ in $Is(H')\cdot Z(H)/Is(H)$ is $d$ by the choice of the $d_i$ and since $d=d_{i_1+1}\cdots d_{i_1+n}$. Now we can choose the number $e$ used in (5) of Lemma~\ref{mainlem} to be the number $e$ introduced here. Then it is clear that $d$ should be relatively prime to $e$, i.e. gcd$(d,e)=1$. 
 
 \end{proof}
\begin{rem} The number $p-n\geq 0$ reflects the maximal rank of a free abelian subgroup of $Z(G)$ (or $G_0$) that splits from $G$. Clearly, from the arguments above this number remains the same when we move from $G$ to a finitely generated elementarily equivalent copy $H$ of it, i.e the number $p-n$ is an elementary invariant of $G$ with respect to finitely generated models of $Th(G)$. \end{rem} 

Before stating some immediate corollaries of this result let us give a name to the construction above. This naming will be justified in the next section.

\begin{defn}[Finitely Generated Abelian Deformations] \label{abdef:defn1}Given a group $G$ presented as in Lemma~\ref{mainlem} and integers $d_i$ and $c_{ik}$ satisfying 3(b) and 3(c) of Proposition~\ref{maincor} the group $H$ constructed thus is called an abelian deformation of $G$.\end{defn}

Theorem~\ref{prop:embed} summarizes some of important implications of the results above.

\noindent\emph{Proof of Theorem~\ref{prop:embed}.} Pick $\bar{u}$ in $G$ and $v_1, \ldots ,v_{i_1},w_{i_1+1}, \ldots , w_{i_2}, v_{i_2+1}, \ldots, v_m$ in $H$ as in Proposition~\ref{maincor}. Then
$u_i\mapsto v_i$ if $i\neq i_1+1, \ldots, i_2$ and $u_i\mapsto w_i$ if $i=i_1+1, \ldots, i_2$ extends to a monomorphism $\phi$ of groups. Statements (a)-(d) are corollaries of the existence and structure of the monomorphism $\phi$. Statement (e) follows from the uniform interpretability (and has already been used in the proofs of the main results here) of this quotient in $G$.\\
\qed
 
\begin{thm}\label{tame:prop} If $G$ is a finitely generated nilpotent group with tame special gap, i.e. $Z(G) \leq Is(G')$, and $H$ is a finitely generated group elementarily equivalent to $G$ then $G\cong H$.\end{thm}
\begin{proof} With $\bar{u}$ and $\bar{v}$ chosen above if $p=i_2-i_1=0$ then the two pseudo-bases have the same length, periods and structure constants. So $G\cong H$.\\
\end{proof}
\noindent\emph{Proof of Theorem~\ref{regular:prop}.} By part (e) of Theorem~\ref{prop:embed} and Proposition \ref{rproperties:prop}, $H$ is regular. Now the statement follows from part (b) of Theorem~\ref{prop:embed}.\\
\qed

\section{Elementary equivalence of finitely generated nilpotent groups and the second cohomology group}\label{elemcohom:sec}
Our aim here is to prove the cohomological form, Theorem~\ref{mainthm}, of Proposition~\ref{maincor}. Here we follow the notation introduced in Lemma~\ref{mainlem}. Let us also recall the following notation
 \begin{align*}
 N(G)&=Is(G')\cdot Z(G)=Is(G')\times G_0\\
 M(G)&=Is(G'\cdot Z(G))\\
 \bar{G}&= \frac{G}{N(G)}
  \end{align*}
   We look at $G$ as the following extension, where $\mu$ is inclusion and $\pi$ is the canonical epimorphism:
\begin{equation}\label{ext:G} 1\rightarrow N(G) \xrightarrow{\mu} G \xrightarrow{\pi} \bar{G} \rightarrow 1.\end{equation}

 Assume $\chi$ is the coupling associated to the extension. Now by Fact~\ref{co} we know that
$G$ corresponds to a unique element in $H^2(\bar{G}, Z(N(G)))$. Note that $N(G)=Is(G')\times G_0$ and thus $Z(N(G))=N_1(G)\times G_0$ where $N_1(G)=Z(Is(G'))$. Since both $N_1(G)$ and $G_0$ are normal in $G$ the action of $G$ by conjugation on $N(H)$ fixes both of the direct factors. So the decomposition $Z(N(G))=N_1(G)\times G_0$ is a direct product of $\bar{G}$-submodules. This implies that a 2-cocycle $f\in Z^2(\bar{G}, Z(N(G))$ can be uniquely written as $f_1\oplus f_2$, modulo 2-coboundaries,  where $f_1\in Z^2(\bar{G},N_1(G))$ and $f_2\in Z^2(\bar{G}, G_0)$. We note that $f_2$ is the 2-cocycle whose class determines the extension:
\begin{equation}\label{ext:G:IsG'} 1\to G_0 \to \frac{G}{Is(G')}\to \bar{G}\to 1,\end{equation}
so indeed $f_2$ is a symmetric 2-cocycle. So we have shown here that 
 \begin{equation}\label{cohom-split1} H^2(\bar{G}, Z(N(G))\cong H^2(\bar{G}, N_1(G))\oplus Ext(\bar{G}, G_0)\end{equation}
 Moreover,

$$\bar{G}=\frac{G}{N(G)} \cong \frac{G}{M(G)} \times \frac{M(G)}{N(G)}.$$
Recall that for abelian groups $A$ and $B$, $Ext(A,B)=0$ whenever $A$ is free abelian. So 
\begin{equation}\begin{split}\label{reduction} Ext(\bar{G}, G_0)&\cong Ext(\frac{G}{M(G)}\times \frac{M(G)}{N(G)}, G_0)\\
&\cong Ext(\frac{G}{M(G)}, G_0)\oplus Ext(\frac{M(G)}{N(G)}, G_0)\\
&\cong Ext(\frac{M(G)}{N(G)}, G_0).
\end{split}\end{equation}

So~\eqref{cohom-split1} actually becomes 

 \begin{equation}\label{cohom-split2}\begin{split}
 H^2(\bar{G}, Z(N(G)))&\cong H^2(\bar{G}, N_1(G))\oplus Ext(\frac{M(G)}{N(G)}, G_0)\\
 \\
   [f]&\mapsto [f_1] \oplus [f_2],\end{split}\end{equation}

Where $[f]$ denotes the class of $f$ in the corresponding second cohomology group. We can also write more elaborate version of \eqref{reduction} and \eqref{cohom-split2} using the integers $e_{i_1+1}, \ldots, e_{i_2}$ in Lemma~\ref{mainlem}. Indeed:

   \begin{equation}\begin{split}\label{abcyc} Ext(\frac{M(G)}{N(G)}, G_0)&\cong Ext(\bigoplus_{i=i_0+1}^{i_1}\langle u_iN(G)\rangle, G_0)\\
   &\cong \bigoplus_{i=i_0+1}^{i_1}Ext(\langle u_jN(G)\rangle, G_0)\end{split}\end{equation}
   
 Now considering \eqref{cohom-split2} and Proposition~\ref{maincor} the reader might suspect that a finitely generated group $H\equiv G$ is isomorphic to an extension of $N(G)$ by $\bar{G}$ whose associated element in $H^2(\bar{G}, Z(N(G))$ is the class of a 2-cocycle $f_1\oplus f_2'$ where $f_1$ is the $f_1$ picked above but $\ds f'_2\in S^2(\frac{M(G)}{N(G)},G_0)$ might belong to a different class compared to $f_2$. 

We have not yet explicitly described how structure constants, which clearly define extensions, relate to 2-cocycles, which are other means of defining extensions. This correspondence is precisely the so-called correspondence between covering group theory and cohomology theory which is discussed in detail in chapter 11 of \cite{robin}. But we won't need to delve into that theory in detail here. Indeed in the case of abelian-by-(finite cyclic) extensions the correspondence can be described rather easily and that is all we need here. The reader might suspect all we need is a description of the symmetric 2-cocycle $f_2$ (or $f_2'$) in \eqref{cohom-split2} in terms of the structure constants defining an abelian extension of $G_0$ by $\bar{G}$. By \eqref{abcyc} this reduces to abelian-by-(finite-cyclic) extensions. Indeed by \eqref{abcyc}, $f_2$ can be uniquely written (modulo the 2-coboundaries) as, 
\begin{equation}\label{f2i} f_2=\sum_{i=i_0+1}^{i_1} f_{2i},\qquad  f_{2i}\in S^2(\langle u_iN(G)\rangle , G_0).  \end{equation}

The following lemma is a special case of exercise 11.1.5 in~\cite{robin}.

\begin{lem}\label{stcyc} Given an abelian group $A$ freely generated by the $\{u_1, \ldots, u_p\}$, integers $c_1, \ldots , c_p$ and an integer $e\geq 2$, the abelian group presented by
$$E=\langle u_1, u_2,\ldots, u_p, g: g^e=u_1^{c_1}\cdots u_p^{c_p}\rangle$$
can be described as an extension 
$$1\to A \to E \to \langle gA\rangle \to 1,$$ whose associated element in $Ext(\langle gA\rangle, A)$ is represented by $f$ defined below.
 \begin{equation*}
f(g^{s}A, g^tA) = \left\{
\begin{array}{lr}
1 & \text{if }~~ s+t<e\\
u_1^{c_1}\cdots u_p^{c_p} & \text{if }~~ s+t\geq e
\end{array} \right.
\end{equation*}\end{lem}

\begin{lem}\label{maincohom:lem} The group Abdef$(G,\bar{d},\bar{c})$ given in Definition~\ref{abdefcohom:defn} is well-defined and defines the same object as the one introduced in Definition~\ref{abdef:defn1}, and therefore if $H$ is finitely generated group $H\equiv G$ then $H\cong \text{Abdef}(G,\bar{d},\bar{c})$ for some integers $d_i$ and $c_{ik}$ as in the definition.\end{lem}

 \begin{proof} Well-definedness follows from considerations resulting in \eqref{cohom-split2} and Lemma \ref{stcyc}. The isomorphism $H\cong \text{Abdef}(G,\bar{d},\bar{c})$ and the claim that the two definitions meet both are direct consequences of Proposition~\ref{maincor} and Lemma~\ref{stcyc}.
 
  \end{proof}

\noindent \emph{Proof of Theorem~\ref{mainthm}.} Obvious from the above lemma.\\ \qed

It is already known that up to isomorphism there are only finitely many finitely generated groups elementarily equivalent to a given finitely generated nilpotent group $G$, say as a corollary of the fact that elementarily equivalent finitely generated nilpotent groups have isomorphic finite quotients and Pickel's theorem that there only finitely many non-isomorphic such groups with isomorphic finite quotient. However from our results this can be seen quite easily.

\begin{cor}\label{number} Given a finitely generated nilpotent group $G$, presented an in Lemma~\ref{mainlem} there are, up to isomorphism, at most $e^p$ finitely generated groups elementarily equivalent to it, where $e=\left|\frac{M(G)}{N(G)}\right|$, and $p$ is the rank of $G_0$.\end{cor}
\begin{proof} Indeed by Theorem~\ref{mainthm} there are at most $|Ext(\frac{M(G)}{N(G)},G_0)|$ such groups up to isomorphism. But by \eqref{abcyc}

\begin{align*}  Ext(\frac{M(G)}{N(G)},G_0)&\cong \bigoplus_{i=i_0+1}^{i_1}Ext(\langle u_iN(G)\rangle, G_0)\\
&\cong\bigoplus_{i=i_0+1}^{i_1}Ext(\frac{\Z}{e_i\Z}, \Z^p)\\
&\cong\bigoplus_{i_0+1}^{i_1}\frac{\Z^p}{e_i\Z^p}\end{align*}

So there are at most $e^p$ number of non-isomorphic groups as such.

\end{proof}

\section{The converse of the characterization theorem}\label{converse:sec}

In this subsection we prove the converse of Characterization Theorem. 

 Let us fix some notation first. Let $\D$ be a non-principal ultrafilter on an index set $I$. By $G^*$ for a group $G$ we mean the ultrapower $G^I/\D$ of $G$. The equivalence class of $x \in G^I$ in $G^*$ is denoted by $\widetilde{x}$. 
\begin{lem}\label{equality:lem} Let $G$ be a finitely generated nilpotent group, and let $\D$ be a non-principal ultrafilter on $I$. Then
\begin{enumerate}
\item $(G')^*= (G^*)'$,
\item $(Is(G'))^*=Is((G^*)')=Is((G')^*),$
\item $Z(G^*)=(Z(G))^*$,
\item If $G_0$ is an addition of $G$ then $(G_0)^*$ is an addition of $G^*$.
\end{enumerate}\end{lem}
\begin{proof} For (1) the inclusion $\geq$ follows from the fact that $(G^I)'$ is generated by $[x,y]=z$, $x, y\in G^I$, where $z(i)=[x(i),y(i)]\in G'$ for $\D$-almost every $i\in I$. But then the equivalence class of $[x,y]$ is in $(G')^*$. The other inclusion follows from the fact that $G'$ is of finite width. Indeed, let $x\in G^I$ be such that the equivalence class $\tilde{x}\in (G')^*$ and assume that width of $G'$ is $n$. Then for $\D$-almost every $i\in I$, $x(i)=[y_1(i), z_1(i)]\cdots[y_n(i),z_n(i)].$ Define $x_j\in G^I$, $j=1,\ldots n$ by $x_j(i)=[y_j(i),z_j(i)]$. Obviously $\widetilde{x_j}\in (G^*)'$, $j=1, \ldots ,n$, and $\tilde{x}=\widetilde{x_1}\cdots \widetilde{x_n}$. This implies the result.      

For (2) the equality of the last two terms follows from (1.). Moreover
$$\tilde{x}\in Is((G')^*)\Leftrightarrow \tilde{x}^m\in (G')^* \Leftrightarrow \widetilde{x^m}\in (G')^* \Leftrightarrow \tilde{x}\in (Is(G'))^*.$$ 
(3) is clear. (4) is implied by (2) and (3) and the definition of an addition. 
\end{proof}
Again we stick with the notation introduced in Lemma~\ref{mainlem} and Proposition~\ref{maincor}.

\begin{lem}\label{satabelian:lem} Assume $A$ is a free abelian group of rank $n$ with basis $v_1, \ldots ,v_n$. Let $A^*$ the ultrapower of $A$ over an $\aleph_1$-incomplete ultrafilter $\D$ and let
\begin{itemize}
\item  $(c_{ij})$ be an $n\times n$ matrix with integer entries and $det((c_{ij}))=\pm 1$
\item  $\alpha_i$, $i=1, \ldots, n$, be elements of the ultrapower $\Z^*$ of the ring of integers $\Z$ over $\D$, such that $p\nmid \alpha_i$ for any prime number $p$, where $|$ denotes division in the ring $\Z^*$.    
\end{itemize} Then there is an automorphism $\psi:A^* \to A^*$ extending $v_i\mapsto \prod_{k=1}^nv_k^{c_{ik}\alpha_k}$. \end{lem} 
\begin{proof} Recall that $A^*$ is an $\aleph_1$-saturated abelian group. By the structure theory of saturated abelian groups (see either of~\cite{Szmielew} or~\cite{eklof}) there is an automorphism $\eta$ of $A^*$ such that $\eta(v_k)=v_k^{\alpha_k}$. Note that the automorphism $\eta$ is not necessarily a $\Z^*$-module automorphism. However since $det(c_{ik})=\pm 1$ there is $\Z^*$-module automorphism of $A^*$ extending $v_i\mapsto \prod_{k=1}^nv_k^{c_{ik}}$. This proves the statement.\end{proof} 
 
\begin{thm}\label{converse} Assume $G$ is a finitely generated nilpotent group presented as in Lemma~\ref{mainlem} and Lemma~\ref{stb}. If $H=Abdef(G,\bar{d}, \bar{c})$ is any abelian deformation of $G$ then   $$G\equiv H.$$\end{thm}
\begin{proof}
In order to prove the statement we prove that ultra-powers $G^*=G^\mathbb{N}/\D$ and $H^*=H^{\mathbb{N}}/\D$ of $G$ and $H$ over any $\omega_1$-incomplete ultrafilter $(\mathbb{N},\D)$ are isomorphic.
 
 Assume $\bar{u}$ and $\bar{v}$ are the pseudo-bases of $G$ and $H$ respectively described in Lemma~\ref{stb} and Proposition~\ref{maincor}. Recall that $e=e_{i_0+1}\cdots e_{i_1}$, $d=d_{i_1}\cdots d_{i_1+n}$ and gcd$(d,e)=1$. Assume $\pi$ denotes the set of all prime numbers and that, $\pi_k$ is the set of all prime numbers $p$, such that $p|d_k$, $i_1+1\leq k \leq i_1+n$. Let us denote the $j$'th prime number in $\pi\setminus \pi_k$ by $p_{kj}$ and the product of the first $j$ primes in $\pi\setminus \pi_k$ by $q_{kj}$.
 
 For each $j\in \N$ and for all $i_0+1\leq i \leq i_1$ define 
 $$w_{ij}=\prod_{k=i_1+1}^{i_1+n}v_k^{c_{ik}(d_k+q_{kj}e)} \in im(\phi_j).$$
 Now let, $w_i^*\in H^*$ and $q_k^*\in \Z^*$ denote the classes of $(w_{ij})_{j\in \N}$ and $(q_{kj})_{j\in \N}$ respectively. Indeed  $$w^*_i=\prod_{k=i_1+1}^{i_1+n}v_k^{c_{ik}(d_k+q_k^*e)}.$$ Let us set $\alpha_k=d_k+q_k^*e$ for each relevant $k$.
 
 Next we claim that the $\alpha_k$ satisfy hypothesis (b) of Lemma~\ref{satabelian:lem}, that is, no prime $p$ divides $\alpha_k=d_k+q_k^*e$ for each $k$,  $k=i_1+1, \ldots, i_1+n$. To prove this we recall that $q_{kj}=p_{k1}\cdots p_{kj}$ where the $p_{k1}, \ldots ,p_{kj}$ are the first $j$ primes that do not divide $d_k$. Pick a prime $p$. If $p\in \pi_d$, i.e. $p|d_k$ and $p|(d_k+q_{kj}e)$, then $p|q_{kj}e$ which contradicts the choice of $q_{kj}$ and the fact that gcd$(d_k,e)=1$. So for such $p$, $p\nmid (d_k+q_{kj}e)$. Now pick a prime $p\in \pi\setminus \pi_k$, i.e $p\nmid d_{k}$. Then $p=p_{kt}$ for some $t\in \N$, meaning that $p$ is a factor of $q_{kj}$ for every $j\geq t$. So $p|q_{kj}e$ for every $j\geq t$. Therefore, for every such $j$ if $p|(d_k+q_{kj}e)$ then $p|d_k$, which is impossible. So for every $j\geq t$, $p\nmid d_k+q_{kj}e$. So indeed for any prime $p$, $p\nmid (d_k+q_{k}^*e)$.

 Let $G_0$ ($H_0$) be the addition of $G$ ($H$) generated by $u_i$ ($v_i$), $i=i_1+1, \ldots , i_2$. By Lemma~\ref{equality:lem} the $\Z^*$-submodule $G^*_0$ ($H^*_0$) of $Z(G^*)$ ($Z(H^*)$) generated be the $u_i$ ($v_i$), $i=i_1+1, \ldots , i_2$ is an addition of $G^*$ ($H^*$) and $G^*_0=(G^*)_0=(G_0)^*$ (the same in $H$). By Lemma~\ref{satabelian:lem} there exists an isomorphism $\psi:G^*_0\to H^*_0$ extending $u_k \to v_k$ if  $i_1+n< k\leq i_2$ and $u_k\to w^*_k$, $i_1<k \leq i_1+n$. 
 
 By construction there exists a monomorphism $\phi:G \to H$ of groups such that: 
  \begin{equation*} 
   \phi(u_i)=\left\{
  \begin{array}{ll}
  v_i & \text{if } i\neq i_1+1,\ldots i_1+n\\ \\
  \prod_{k=i_1+1}^{i_1+n}v_k^{d_kc_{ik}} & \text{if }i= i_1+1,\ldots i_1+n \end{array}\right.
   \end{equation*}
  Actually $\phi$ defined above is the same $\phi$ as in Proposition~\ref{prop:embed} while the $w_i$ are written with respect to the new basis of $H_0$ found in the proof Proposition~\ref{maincor}.  For each $j\in \N$ one could twist the monomorphism $\phi:G\to H$ to get a new one denoted by $\phi_j:G\to H$ and defined by: 
 \begin{equation*} 
   \phi_j(u_i)=\left\{
  \begin{array}{ll}
  v_i & \text{if }i\neq i_0+1,\ldots,i_1,\ldots , i_1+n\\ \\
  v_i\prod_{k=i_1+1}^{i_1+n}v_k^{q_{kj}\hat{e}_ic_{ik}}& \text{if }i= i_0+1,\ldots ,i_1\\ \\
  \prod_{k=i_1+1}^{i_1+n}v_k^{d_kc_{ik}+q_{kj}ec_{ik}} & \text{if }i= i_1+1,\ldots, i_1+n \end{array}\right.
   \end{equation*}
where $\hat{e}_i=e/e_i$. Again note that $G$ and $im(\phi_j)$ are generated by the pseudo-bases of the same lengths, periods and structure constants. Let $\phi^*:G^*\to H^*$ be the monomorphism induced $(\phi_j)_{j\in \N}$. 

Consider the subgroup $G^*_f$ of $G^*$ generated by $$\{u_i^{\alpha}, u_j: i\neq i_0+1,\ldots,i_1,\ldots , i_1+n, \alpha\in \Z^*, j=i_0+1,\ldots, i_1\}.$$ We claim  that $G^*=G^*_f \cdot G^*_0$. Firstly $G^*$ is generated by all the $u_i^\alpha$, $\alpha\in \Z^*$, exponentiation defined in an obvious manner. All these generators belong to $G^*_f\cdot G^*_0$ with the possible exceptions of $i=i_0+1, \ldots i_1$. However by  Lemma~\ref{equality:lem} 
$$\frac{Is((G^*)'\cdot Z(G^*))}{Is((G^*)')\cdot Z(G^*)}\cong \frac{Is(G'\cdot Z(G))}{Is(G')\cdot Z(G)}$$
is a finite abelian group and so we only need integer powers of the $u_i$, $i=i_0+1, \ldots i_1$ in the generating set. This proves the claim.
  
Now given $x\in G^*$ there are $y\in G^*_f$ and $z\in G^*_0$ such that $x=yz$. Now define a map $\eta:G^* \to H^*$ by 
$$\eta(x)= \phi^*(y)\psi(z).$$ 
To show that $\eta$ is well-defined we need to check if $\phi^*$ and $\psi$ agree on $G^*_f\cap G^*_0$, which is clear by definition of the maps $\phi^*$ and $\psi$. $\eta$ is a homomorphism since $\phi^*$ and $\psi$ are and $\psi$ maps a central subgroup of $G^*$ into a central subgroup of $H^*$. It is injective since both $\phi^*$ and $\psi$ are injective. Finally  $H^*=\phi^*(G^*_f)\cdot H^*_0$ and by construction $H^*_0=im(\psi)$. Therefore $\eta$ is surjective. We have proved that $\eta: G^* \to H^*$ is an isomorphism of groups and by the Keisler-Shelah's theorem, we have proved that
 $$G \equiv H.$$ 

\end{proof}

\section{Zilber's example}\label{Zilber:sec}

Here we present an example due to B. Zilber~\cite{Z71} which was the first example of two finitely generated nilpotent groups $G$ and $H$ where $G\equiv H$ but $G\ncong H$. We shall show that $G$ and $H$ are best described as abelian deformations of one another.

Consider the groups $G$ and $H$ presented (in the category of 2-nilpotent groups) as follows.
$$G=\langle a_1,b_1,c_1,d_1|a^5_1 \text{ is central}, [a_1,b_1][c_1,d_1]=1\rangle,$$
$$H=\langle a_2,b_2,c_2,d_2|a^5_2 \text{ is central}, [a_2,b_2]^2[c_2,d_2]=1\rangle.$$
B. Zilber~\cite{Z71} proved that $G\equiv H$ but $G\ncong H$. Let us first apply a Titze transformation to both $G$ and $H$ to get
$$G\cong \langle a_1,b_1,c_1,d_1,f_1|f_1 \text{ is central}, a_1^5f^{-1}_1=1, [a_1,b_1][c_1,d_1]=1\rangle$$
$$H\cong \langle a_2,b_2,c_2,d_2,f_2|f_2 \text{ is central}, a_2^5f_2^{-1}=1, [a_2,b_2]^2[c_2,d_2]=1\rangle.$$
Now we are going to show that $H$  is an abelian deformation of $G$. So define a group $K$ by
$$K= \langle a_3,b_3,c_3,d_3,f_3|f_3 \text{ is central}, a_3^5f^{-2}_3=1, [a_3,b_3][c_3,d_3]=1\rangle.$$
Note that we can choose $G_0=\langle f_1=a_1^5\rangle$ and $K_0=\langle f_3\rangle$. We also have that $Is(G'\cdot G_0)=G'\cdot Is(G_0),$ and
$Is(K'\cdot K_0)=K'\cdot Is(K_0)$, and
so $$Is(G'\cdot G_0)/Is(G')\cdot G_0=Is(G_0)/G_0=\langle a_1|a_1^5=1\rangle$$ and $$Is(K'\cdot K_0)/Is(K')\cdot K_0=Is(K_0)/K_0=\langle a_3|a_3^5=1\rangle.$$
Indeed using notation of Definition~\ref{abdefcohom:defn}, $K=\text{Abdef}(G,d,c)$ where $d=2$ and $(c)$ is $1\times 1$ matrix (1). Indeed we only changed the structure constant $t_{f_1}(a_1^5)$ to deform  $G$ to $K$. However in $K$
$$(a_3^3f^{-1}_3)^2=a_3^6f_3^{-2}=a_3(a_3^5f_3^{-2})=a_3.$$
Now apply the corresponding Titze transformation to the presentation of $K$ to get
$$K= \langle a^3_3f_3^{-1},b_3,c_3,d_3,f_3|f_3 \text{ is central}, (a_3^3f_3^{-1})^5f_3^{-1}=1, [a_3^3f_3^{-1},b_3]^2[c_3,d_3]=1\rangle.$$
Obviously $H\cong K.$



\end{document}